\documentclass[12pt]{amsart}
\usepackage{amsthm}
\usepackage{amsmath,amssymb,euscript,mathrsfs,bm}

\usepackage{pstricks}

\usepackage{eepic,bbm}

\usepackage{bbold}
\usepackage{dsfont}

\usepackage{ebgaramond}

\usepackage{setspace}
\usepackage[shortlabels]{enumitem}

\usepackage{tikz-cd}

\usepackage[colorlinks=true,linkcolor=black,citecolor=black,urlcolor=black]{hyperref}

\pushQED{\qed}

\psset{unit=1.5pt}
\psset{arrowsize=4pt 1}
\psset{linewidth=.5pt}

\countdef\x=23
\countdef\y=24
\countdef\z=25
\countdef\t=26

\def\graybox(#1,#2){
\x=#1 \y=#2 
\z=\x \t=\y
\advance\z by 10 
\advance\t by 10 
\psframe[fillstyle=solid,fillcolor=lightgray,linewidth=0pt](\x,\y)(\z,\t) 
\psline[linewidth=.5pt](\x,\y)(\x,\t)(\z,\t)(\z,\y)(\x,\y)}

\def\emptygraybox(#1,#2){
\x=#1 \y=#2 
\z=\x \t=\y
\advance\z by 10 
\advance\t by 10 
\psframe[fillstyle=solid,fillcolor=lightgray,linewidth=0pt,linecolor=lightgray](\x,\y)(\z,\t)}

\def\blankbox(#1,#2){
\x=#1 \y=#2 
\z=\x \t=\y
\advance\x by 1
\advance\y by 1
\advance\z by 9 
\advance\t by 9
\psline[linewidth=1pt](\x,\y)(\x,\t)(\z,\t)(\z,\y)(\x,\y)}

\newcommand{\define}{\textbf}

\newcommand{\C}{\mathds{C}}

\newcommand{\Z}{\mathds{Z}}

\newcommand{\ba}{{\bm{a}}}
\newcommand{\bb}{{\bm{b}}}
\newcommand{\bd}{{\bm{d}}}

\newcommand{\cR}{\mathcal{R}}

\DeclareMathOperator{\Sym}{Sym}

\DeclareMathOperator{\rk}{rk}
\DeclareMathOperator{\diag}{diag}
\DeclareMathOperator{\tchar}{char}

\DeclareMathOperator{\Hom}{Hom}

\DeclareMathOperator{\Proj}{Proj}

\newcommand{\tto}{\twoheadrightarrow}

\newcommand{\bull}{ {\scriptscriptstyle{\ensuremath{\bullet}}}  }

\newcommand{\dual}{*}      

\newcommand{\kk}{\mathds{k}}        

\newcommand{\Rop}{\mathsf{R}}   
\newcommand{\Top}{\mathsf{T}}   
\newcommand{\Zop}{\mathsf{Z}}   

\newcommand{\fS}{\mathfrak{S}}        

\newtheoremstyle{scthm}%
{}{}{\itshape}{}{\scshape}{.}{ }{}

\newtheoremstyle{scdef}%
{}{}{}{}{\scshape}{.}{ }{}

\theoremstyle{scthm}

\newtheorem{theorem}{Theorem}[section]
\newtheorem{lemma}[theorem]{Lemma}
\newtheorem{proposition}[theorem]{Proposition}
\newtheorem{corollary}[theorem]{Corollary}
\newtheorem{conjecture}[theorem]{Conjecture}

\newtheorem*{thm*}{Theorem}
\newtheorem*{cor*}{Corollary}
\newtheorem*{prop*}{Proposition}
\newtheorem*{claim*}{Claim}

\theoremstyle{scdef}
\newtheorem{definition}[theorem]{Definition}
\newtheorem{remark}[theorem]{Remark}
\newtheorem{example}[theorem]{Example}

\begin{document}

\newcommand{\isom}{\cong}
\renewcommand{\setminus}{\smallsetminus}
\renewcommand{\phi}{\varphi}
\newcommand{\exterior}{\textstyle\bigwedge}
\renewcommand{\tilde}{\widetilde}
\renewcommand{\hat}{\widehat}
\renewcommand{\bar}{\overline}

\title{Filtrations and recursions for Schubert modules}
\author{David Anderson}

\date{February 22, 2026}
\address{Department of Mathematics, The Ohio State University, Columbus, OH 43210}
\email{anderson.2804@math.osu.edu}
\thanks{Partially supported by NSF CAREER DMS-1945212 and by a Simons-CRM scholarship.}

\maketitle

\renewcommand{\bfseries}{\itshape}

\begin{abstract}
Revisiting Kra\'skiewicz and Pragacz's construction of {\it Schubert modules}, we provide a new proof that their characters are equal to Schubert polynomials.  The main innovation is a representation-theoretic interpretation of a recurrence relation for Schubert polynomials recently discovered by Nadeau, Spink, and Tewari.  Along the way, we review several related constructions, and show that the Nadeau-Spink-Tewari recursion determines the characters of flagged Schur modules coming from the broader classes of {\it transparent} and {\it translucent} diagrams.  We conclude with a conjecture concerning the Schubert positivity of the characters of transparent diagrams.
\end{abstract}

\setcounter{tocdepth}{1}
\tableofcontents

\section*{Introduction}

Schur functions appear in geometry as representatives for the cohomology classes of Schubert varieties in the Grassmannian, and in representation theory as the characters for irreducible $GL_n$-modules.  They form a distinguished basis of the ring of symmetric polynomials, and are among the most-studied objects in algebraic combinatorics.

Schubert polynomials were introduced in the context of geometry, as canonical representatives for the cohomology classes of Schubert varieties in the complete flag variety \cite{ls}.  They form a distinguished basis for the ring of all polynomials, and they also have wonderful combinatorial properties, in many ways analogous to those of Schur polynomials.

A natural construction of modules having Schubert polynomials as their characters was found by Kra\'skiewicz and Pragacz.  This time the construction produces cyclic $B$-modules, where $B$ is a Borel subgroup of upper (or lower) triangular matrices \cite{kp,kp2}.  These $B$-modules were investigated rather extensively around 25 years ago by Magyar \cite{magyar,magyar-bw} and Reiner-Shimozono \cite{rs}; in the decades that followed, they received intermittent attention (see, e.g., \cite{watanabe-positive,watanabe}).  In the last few years, though, there seems to be renewed interest in these constructions: they have been revisited combinatorially, and some of the resulting formulas have been extended to Grothendieck polynomials (e.g., \cite{mss,mst,fg} and others).

On the other hand, Nadeau, Spink, and Tewari \cite{nst,nst2} have recently found a wonderful collection of identities arising from operators on the polynomial ring, including a remarkable one for Schubert polynomials:
\begin{equation}\label{e.nst1}
   \fS_w = \Rop_1(\fS_{w}) + \sum_{k\in\mathrm{Des}(w)} x_k \cdot \Rop_{k+1}( \fS_{ws_k} ).
\end{equation}
The operators $\Rop_k$ act on polynomials by mapping $x_k\mapsto 0$ and $x_i\mapsto x_{i-1}$ for $i>k$.  Nadeau, Spink, and Tewari call them \define{Bergeron-Sottile operators}, in reference to their connection with direct-sum embeddings of flag varieties \cite{bergeron-sottile}.

The results presented here are motivated by a desire to understand \eqref{e.nst1} as arising naturally from the representation theory of Kra\'skiewicz-Pragacz's Schubert modules, here denoted $E_\bull^{D(w)}$.  The notation refers to a quotient flag $E_\bull: E_n \tto E_{n-1} \tto \cdots \tto E_1$, with $\dim E_i = i$, and to the Rothe diagram $D(w)$ of the permutation.  In brief, we construct a functorial filtration
\begin{equation}\label{e.filtration}
  E_\bull^{D(w)} = \Rop_{n+1}E_\bull^{D(w)} \tto \Rop_nE_\bull^{D(w)} \tto \cdots \tto \Rop_1E_\bull^{D(w)}
\end{equation}
such that $\Rop_{k+1}E_\bull^{D(w)} \tto \Rop_kE_\bull^{D(w)}$ is an isomorphism unless $k$ is a (right) descent of $w$, in which case the homomorphism has kernel isomorphic to $\ker(E_k \tto E_{k-1}) \otimes \Rop_{k+1}E_\bull^{D(ws_k)}$.  This implies a formula in K-theory of $B$-modules:
\begin{equation}\label{e.module-recursion}
  [E_\bull^{D(w)}] = [\Rop_1E_\bull^{D(w)}] + \sum_{k\in\mathrm{Des}(w)} [\ker(E_k \tto E_{k-1})] \cdot [\Rop_{k+1}E_\bull^{D(ws_k)}].
\end{equation}
Since \eqref{e.nst1} recursively determines the Schubert polynomials, \eqref{e.module-recursion} provides an independent proof that $\fS_w$ is the character of $E_\bull^{D(w)}$.  
The main ingredients---Lemmas~\ref{l.k-full} and \ref{l.k-descent}---rely only on simple facts from classical invariant theory. 

The construction of $E_\bull^{D(w)}$ extends naturally to produce a $B$-module $E_\bull^D$ for any {\it diagram} $D$, meaning a finite subset of boxes in $\Z_{>0}^2$ (usually displayed in matrix coordinates).  As a byproduct of the discussion in \S\ref{s.kernel} and \S\ref{s.recursion}, we find that a recursion analogous to that of \eqref{e.nst1} determines the characters of $E_\bull^D$ whenever $D$ is {\it transparent} or {\it translucent} (Definitions~\ref{d.transparent} and \ref{d.translucent}).  The recursive formula can be expressed as a sum over {\it reduced words} for such diagrams, for a suitable notion of reduced word (Corollary~\ref{c.redword}).

To my knowledge, the most thorough investigation of such modules and their characters was carried out by Reiner-Shimozono \cite{rs}, who obtained precise combinatorial formulas for the character of $E_\bull^D$ under the hypothesis that $D$ is {\it \%-avoiding}.  Their methods involve a different recursion, coming from Magyar's orthodontic poset.  Examples in \S\ref{s.recursion} show that the notions of translucent and \%-avoiding are distinct, with neither encompassing the other, although both include Rothe diagrams of permutations.

The recursion \eqref{e.nst1} has many pleasant properties---it is strikingly easy to prove, and it is manifestly monomial-positive.  On the other hand, being essentially equivalent to a sum over reduced words, it is not computationally efficient, and the same is true of our extension to more general characters.  (By contrast, Magyar's orthodontic recursion is fast---but not obviously positive, as it involves divided difference operators.)  For diagrams inside a $6\times 6$ rectangle, however, experimentation is feasible.  Julia code for computing characters is available at \cite{schubmods}.

In the course of describing the setup, I take the opportunity to review the definitions and constructions of (flagged) Schur modules for general diagrams, reformulating some of them according to the functorial definitions found in, e.g., \cite{abw} or \cite[\S8.1]{fulton-yt}, and supplying several worked examples. The introductory sections (\S\S\ref{s.schur}--\ref{s.flagged}) include an extended---but far from exhaustive---exposition of these constructions.  One technical point is worth noting: as presented here, the flagged Schur module $E_\bull^D$ is naturally a pushout of the (unflagged) module $E^D$ and a tensor product of flagged column modules $TE_\bull^D$ (Equation~\eqref{e.pushout-flag}).  This lets us deploy a basic fact about pushouts to describe kernels of maps in the filtration \eqref{e.filtration} (Lemma~\ref{l.ker-gen}).

The new ideas are introduced with the Bergeron-Sottile operators in \S\ref{s.Rk}, and the reader familiar with the relevant background is invited to start there.  The main arguments are provided in \S\ref{s.kernel}, with Lemmas~\ref{l.ker-gen}, \ref{l.k-full}, and \ref{l.k-descent}.  The main results appear in \S\ref{s.recursion}---see especially Theorem~\ref{t.recur}, which generalizes \eqref{e.nst1} to characters of modules coming from a broader class of diagrams---and \S\ref{s.apps} contains some further consequences and examples, along with a conjecture about the Schubert-positivity of the characters of transparent diagrams.

\bigskip
\noindent
{\it Acknowledgements.}  This project began when Philippe Nadeau and Vasu Tewari told me about \eqref{e.nst1} during the CAAC workshop in Montreal in January 2024.  Meetings with Hunter Spink and Vasu Tewari at the Fields Institute in April 2024 provided further stimulation.  I thank all three of them for sharing ideas and a lively ongoing conversation.  Thanks also to Hugh Dennin for several helpful discussions, to Vic Reiner and Linus Setiabrata for comments on an earlier draft, and to the referee for corrections and suggestions for improvement.

I am grateful for the pleasant working environment provided by LACIM during a sabbatical visit to UQ\`AM in 2023--24, and some of this work was carried out during a visit to the Institute for Advanced Study funded by the Charles M. Simonyi endowment.

This article is dedicated to Bill Fulton on his 85th birthday---it was from him that I first learned about Schur modules and their functorial construction, along with so much else about being a mathematician.

\section{Schur modules}\label{s.schur}

Let $\kk$ be a commutative ring.  Examples we have in mind usually make $\kk$ a field, often of characteristic zero, but we will not require this in general.

Let $E$ be a finitely generated $\kk$-module.  We will need some basic linear-algebraic constructions generalizing Schur functors, based on the version described by Fulton \cite[\S8.1]{fulton-yt}.  At the same time, we introduce some notation which will be used throughout the paper.

Given a set of positive integers $\ba=\{a_1<\cdots<a_r\}$, we represent $\ba$ as a collection of $r$ boxes stacked in a column, with a box present in rows $a_1,\ldots,a_r$.  Grayscale shading is used as a placeholder to mark an empty row.  For added clarity, we often label the row numbers.  (The row labeling will not be needed until we discuss flagged modules in the next section, and later it will be used to indicate the flagging conditions themselves.)  For example, the set $\ba = \{2,3,5\}$ is represented as follows:

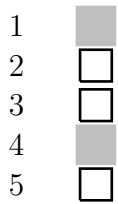
\begin{center}
\pspicture(0,-5)(30,55)

\rput(-5,45){$1$}
\rput(-5,35){$2$}
\rput(-5,25){$3$}
\rput(-5,15){$4$}
\rput(-5,5){$5$}

\emptygraybox(10,40)
\emptygraybox(10,10)

\blankbox(10,30)
\blankbox(10,20)
\blankbox(10,0)

\endpspicture
\end{center}
A tuple of vectors in $E^{\times r}$ can be specified by placing them in the corresponding boxes.  For example, $(u,v,w)$ might be written as
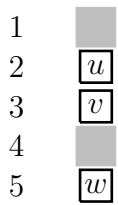
\begin{center}
\pspicture(0,-5)(30,55)

\rput(-5,45){$1$}
\rput(-5,35){$2$}
\rput(-5,25){$3$}
\rput(-5,15){$4$}
\rput(-5,5){$5$}

\emptygraybox(10,40)
\emptygraybox(10,10)

\blankbox(10,30) \rput(15,35){$u$}
\blankbox(10,20) \rput(15,25){$v$}
\blankbox(10,0) \rput(15,5){$w$}

\endpspicture
\end{center}

\subsection{Column modules}

Our first construction is simply $\exterior^r E$, viewed as a quotient of $E^{\otimes r}$.  But in preparation for what is to come, we phrase it as follows:
\begin{definition}
The $\kk$-module $E^\ba$ is defined by the universal property
\[
  \Hom(E^{\ba},M) = \Big\{\text{multilinear maps } \phi\colon E^{\times r} \to M \,\Big|\, \phi(\ldots,v,\ldots,v\ldots) = 0 \text{ for any } v\in E\Big\}
\]
\end{definition}

For instance, with $\ba=\{2,3,5\}$ as above, the relations defining $E^\ba$ are
\begin{center}
\pspicture(50,-5)(70,55)

\rput(0,25){$\phi \,\Big($}

\emptygraybox(10,40)
\emptygraybox(10,10)

\blankbox(10,30) \rput(15,35){$u$}
\blankbox(10,20) \rput(15,25){$v$}
\blankbox(10,0) \rput(15,5){$w$}

\rput[l](25,25){$\Big)=0$ whenever $u=v$, $u=w$, or $v=w$.}
\endpspicture
\end{center}
The module $E^{\ba}$ depends only on $r$, not on $\ba$, but the extra notation will be justified soon.  Further abusing notation, one often writes
\begin{center}
\pspicture(50,-5)(70,55)

\emptygraybox(10,40)
\emptygraybox(10,10)

\blankbox(10,30) \rput(15,35){$u$}
\blankbox(10,20) \rput(15,25){$v$}
\blankbox(10,0) \rput(15,5){$w$}
\endpspicture
\end{center}
to indicate $u\wedge v\wedge w$, the image of $u\otimes v\otimes w$ in $E^\ba$.

When $E$ is free, say with basis $z_1,\ldots,z_n$, we simply record the index to indicate a basis vector in a box.  The module $E^{\ba}$ acquires a basis in the usual way, by writing
\[
  z_{\ba,\bb} = z_{b_1}\wedge \cdots \wedge z_{b_r}
\]
for $b_1<\cdots<b_r$.  These basis elements are depicted by labelling the boxes of $\ba$.  So, continuing our example with $\ba=\{2,3,5\}$, if $E$ is free of rank $5$, the module $E^\ba$ is free of rank $10$, with basis
\begin{center}
\pspicture(0,-5)(30,55)

\emptygraybox(10,40)
\emptygraybox(10,10)

\blankbox(10,30) \rput(15,35){$1$}
\blankbox(10,20) \rput(15,25){$2$}
\blankbox(10,0) \rput(15,5){$3$}

\endpspicture
\pspicture(0,-5)(30,55)

\emptygraybox(10,40)
\emptygraybox(10,10)

\blankbox(10,30) \rput(15,35){$1$}
\blankbox(10,20) \rput(15,25){$2$}
\blankbox(10,0) \rput(15,5){$4$}

\endpspicture
\pspicture(0,-5)(30,55)

\emptygraybox(10,40)
\emptygraybox(10,10)

\blankbox(10,30) \rput(15,35){$1$}
\blankbox(10,20) \rput(15,25){$3$}
\blankbox(10,0) \rput(15,5){$4$}

\endpspicture
\pspicture(0,-5)(30,55)

\emptygraybox(10,40)
\emptygraybox(10,10)

\blankbox(10,30) \rput(15,35){$2$}
\blankbox(10,20) \rput(15,25){$3$}
\blankbox(10,0) \rput(15,5){$4$}

\endpspicture
\pspicture(0,-5)(30,55)

\emptygraybox(10,40)
\emptygraybox(10,10)

\blankbox(10,30) \rput(15,35){$1$}
\blankbox(10,20) \rput(15,25){$2$}
\blankbox(10,0) \rput(15,5){$5$}

\endpspicture
\end{center}

\begin{center}
\pspicture(0,-5)(30,55)

\emptygraybox(10,40)
\emptygraybox(10,10)

\blankbox(10,30) \rput(15,35){$1$}
\blankbox(10,20) \rput(15,25){$3$}
\blankbox(10,0) \rput(15,5){$5$}

\endpspicture
\pspicture(0,-5)(30,55)

\emptygraybox(10,40)
\emptygraybox(10,10)

\blankbox(10,30) \rput(15,35){$2$}
\blankbox(10,20) \rput(15,25){$3$}
\blankbox(10,0) \rput(15,5){$5$}

\endpspicture
\pspicture(0,-5)(30,55)

\emptygraybox(10,40)
\emptygraybox(10,10)

\blankbox(10,30) \rput(15,35){$1$}
\blankbox(10,20) \rput(15,25){$4$}
\blankbox(10,0) \rput(15,5){$5$}

\endpspicture
\pspicture(0,-5)(30,55)

\emptygraybox(10,40)
\emptygraybox(10,10)

\blankbox(10,30) \rput(15,35){$2$}
\blankbox(10,20) \rput(15,25){$4$}
\blankbox(10,0) \rput(15,5){$5$}

\endpspicture
\pspicture(0,-5)(30,55)

\emptygraybox(10,40)
\emptygraybox(10,10)

\blankbox(10,30) \rput(15,35){$3$}
\blankbox(10,20) \rput(15,25){$4$}
\blankbox(10,0) \rput(15,5){$5$}

\endpspicture
\end{center}

\subsection{Diagrams and exchange relations}\label{ss.d-exchange}

A \define{diagram} is a tuple $D = \left[\ba^{(1)},\ba^{(2)},\ldots,\ba^{(m)}\right]$ of sets of positive integers.  We represent a diagram $D= \left[\ba^{(1)},\ba^{(2)},\ldots,\ba^{(m)}\right]$ by placing the columns of boxes $\ba^{(1)}$, $\ba^{(2)}$, etc., alongside one another.  For example, we depict the diagram $D = \Big[ \{2,3\},\;\{2,3,5\},\;\{3\}\Big]$ as follows:
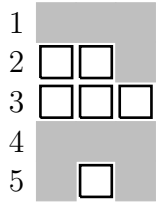
\begin{center}
\pspicture(0,0)(30,60)

\rput(-5,45){$1$}
\rput(-5,35){$2$}
\rput(-5,25){$3$}
\rput(-5,15){$4$}
\rput(-5,5){$5$}

\emptygraybox(0,40)
\emptygraybox(0,10)
\emptygraybox(0,0)

\blankbox(0,30)
\blankbox(0,20)


\emptygraybox(10,40)
\emptygraybox(10,10)

\blankbox(10,30)
\blankbox(10,20)
\blankbox(10,0)

\emptygraybox(20,40)
\emptygraybox(20,30)
\emptygraybox(20,10)
\emptygraybox(20,0)

\blankbox(20,20)

\endpspicture
\end{center}
%
Often we will consider a diagram $D$ as a collection of boxes $(i,j)$, so that the boxes in column $j$ form the set $\ba^{(j)}$.  The meaning should be clear from context.

For our purposes, nothing changes by re-ordering the columns of a diagram, or by adding or omitting empty columns, so we often do this without comment.

Given a diagram $D$, we will define a $\kk$-module $E^D$ by taking a certain quotient of the tensor product
\[
  TE^D = E^{\ba^{(1)}} \otimes \cdots \otimes E^{\ba^{(m)}}.
\]
The relations defining the quotient will be called the \define{$D$-exchange relations}.  In a nutshell, these are all the linear relations among $m$-fold products of minors of a generic matrix, where the $i$-th minor is taken on rows $\ba^{(i)}$.

To describe the $D$-exchange relations more precisely, first assume $E$ is free on $n$ generators, and choose a basis $z_1,\ldots,z_n$.  We consider $Z = \kk^n \otimes E$ as a space of $n\times n$ matrices, with basis $z_{ij} := e_i \otimes z_j$, where $e_1,\ldots,e_n$ is the standard basis of $\kk^n$.  Let $\kk[Z] = \Sym^*(\kk^n \otimes E)$ be the polynomial ring in variables $z_{ij}$.  For sets $\ba = \{a_1<\cdots<a_r\}$ and $\bb=\{b_1<\cdots<b_r\}$, let $\Delta_{\ba,\bb} \in \kk[Z]$ be the minor of $Z$ on rows $\ba$ and columns $\bb$.

A \define{filling} of $D$ is a function $\tau\colon D \to \Z_{>0}$.  A filling is \define{column-strict} if it is (strictly) increasing on columns: whenever $(i,j)$ and $(i',j)$ are in $D$, with $i<i'$, we have $\tau(i,j)<\tau(i',j)$.  
Restricting a column-strict filling $\tau$ to a column $\ba^{(j)}$, we get a set $\bb^{(j)}$.  The module $TE^D$ has a basis indexed by column-strict fillings of $D$,
\[
  z_\tau = z_{\ba^{(1)},\bb^{(1)}} \otimes \cdots \otimes z_{\ba^{(m)},\bb^{(m)}}.
\]

Now we define a homomorphism $TE^D \to \kk[Z]$ by
\[
  z_\tau \mapsto \Delta_\tau := \Delta_{\ba^{(1)},\bb^{(1)}} \cdots \Delta_{\ba^{(m)},\bb^{(m)}}.
\]
For example, if $\tau$ is the filling given by
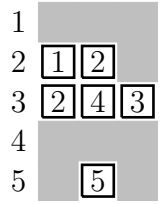
\begin{center}
\pspicture(0,0)(30,60)

\rput(-5,45){$1$}
\rput(-5,35){$2$}
\rput(-5,25){$3$}
\rput(-5,15){$4$}
\rput(-5,5){$5$}

\emptygraybox(0,40)
\emptygraybox(0,10)
\emptygraybox(0,0)

\blankbox(0,30) \rput(5,35){$1$}
\blankbox(0,20) \rput(5,25){$2$}


\emptygraybox(10,40)
\emptygraybox(10,10)

\blankbox(10,30) \rput(15,35){$2$}
\blankbox(10,20) \rput(15,25){$4$}
\blankbox(10,0)  \rput(15,5){$5$}

\emptygraybox(20,40)
\emptygraybox(20,30)
\emptygraybox(20,10)
\emptygraybox(20,0)

\blankbox(20,20)  \rput(25,25){$3$}

\endpspicture
\end{center}
then $z_\tau$ maps to
\[
\Delta_\tau = \left|\begin{array}{cc} z_{21} & z_{22} \\ z_{31} & z_{32} \end{array}\right|
\cdot
\left|\begin{array}{ccc} z_{22} & z_{24} & z_{25} \\ z_{32} & z_{34} & z_{35} \\ z_{52} & z_{54} & z_{55} \end{array}\right|
\cdot
z_{33}
\]
The homomorphism $z_\tau \mapsto \Delta_\tau$ is injective if the sets $\ba^{(j)}$ are pairwise disjoint, but in general there may be a kernel.

\begin{definition}\label{d.D-exchange}
Given a diagram $D$, the \define{$D$-exchange relations} are the elements of the kernel of the map $z_\tau \mapsto \Delta_\tau$.
\end{definition}

We also use the term ``$D$-exchange relation'' to refer to any equation implied by these relations.

If $D$ consists of the two columns $\ba^{(1)}=\ba^{(2)}=\{1,\ldots,r\}$, there are the two-column (quadratic) {\it Sylvester relations}: given a filling $\tau$---that is, a pair of sets $\bb^{(1)},\bb^{(2)}$---select a box in the right-hand column, say in row $i$, with corresponding entry $b^{(2)}_i$.  Then
\begin{equation}
  \Delta_\tau = \sum_{\tau'} \Delta_{\tau'}
\end{equation}
the sum over all ways of exchanging the selected entry on the right, $b^{(2)}_i$, with some entry on the left, $b^{(1)}_{i'}$.  This is a special case of a more general identity due to Sylvester; see \cite[\S8.1, Lemma~2]{fulton-yt}.

For example, there is a relation
\begin{equation}\label{e.2column-relation-maximal}
\begin{aligned}
\raisebox{-50pt}{
\pspicture(0,0)(20,50)

\blankbox(0,40) \rput(5,45){$2$}
\blankbox(0,30) \rput(5,35){$3$}
\blankbox(0,20) \rput(5,25){$4$}

\blankbox(10,40) \rput(15,45){$1$}
\blankbox(10,30)  \rput(15,35){$2$}
\blankbox(10,20)  \rput(15,25){${\bf5}$}

\endpspicture
}
&=
\raisebox{-50pt}{
\pspicture(0,0)(20,50)

\blankbox(0,40) \rput(5,45){${\bf5}$}
\blankbox(0,30) \rput(5,35){$3$}
\blankbox(0,20) \rput(5,25){$4$}

\blankbox(10,40) \rput(15,45){$1$}
\blankbox(10,30)  \rput(15,35){$2$}
\blankbox(10,20)  \rput(15,25){${\bf2}$}

\endpspicture
}
+
\raisebox{-50pt}{
\pspicture(0,0)(20,50)

\blankbox(0,40) \rput(5,45){$2$}
\blankbox(0,30) \rput(5,35){${\bf5}$}
\blankbox(0,20) \rput(5,25){$4$}

\blankbox(10,40) \rput(15,45){$1$}
\blankbox(10,30)  \rput(15,35){$2$}
\blankbox(10,20)  \rput(15,25){${\bf3}$}

\endpspicture
}
+
\raisebox{-50pt}{
\pspicture(0,0)(20,50)

\blankbox(0,40) \rput(5,45){$2$}
\blankbox(0,30) \rput(5,35){$3$}
\blankbox(0,20) \rput(5,25){${\bf5}$}

\blankbox(10,40) \rput(15,45){$1$}
\blankbox(10,30)  \rput(15,35){$2$}
\blankbox(10,20)  \rput(15,25){${\bf4}$}

\endpspicture
}
\\
&= 0 -
\raisebox{-50pt}{
\pspicture(0,0)(20,40)

\blankbox(0,40) \rput(5,45){$2$}
\blankbox(0,30) \rput(5,35){$4$}
\blankbox(0,20) \rput(5,25){$5$}

\blankbox(10,40) \rput(15,45){$1$}
\blankbox(10,30)  \rput(15,35){$2$}
\blankbox(10,20)  \rput(15,25){${3}$}

\endpspicture
}
+
\raisebox{-50pt}{
\pspicture(0,0)(20,40)

\blankbox(0,40) \rput(5,45){$2$}
\blankbox(0,30) \rput(5,35){$3$}
\blankbox(0,20) \rput(5,25){${5}$}

\blankbox(10,40) \rput(15,45){$1$}
\blankbox(10,30)  \rput(15,35){$2$}
\blankbox(10,20)  \rput(15,25){${4}$}

\endpspicture
}\quad.
\end{aligned}
\end{equation}

The same type of relation holds for any pair of identical sets of row-indices $\ba^{(1)}=\ba^{(2)}$, not necessarily top-justified ones.  A similar relation holds whenever there is a containment $\ba^{(1)} \supset \ba^{(2)}$.  For instance, we have

\begin{equation}
\raisebox{-40pt}{
\pspicture(0,0)(20,40)

\blankbox(0,30) \rput(5,35){$1$}
\blankbox(0,20) \rput(5,25){$2$}

\emptygraybox(10,30)

\blankbox(10,20)  \rput(15,25){${\bf3}$}

\endpspicture
}
=
\raisebox{-40pt}{
\pspicture(0,0)(20,40)

\blankbox(0,30) \rput(5,35){$1$}
\blankbox(0,20) \rput(5,25){$3$}

\emptygraybox(10,30)

\blankbox(10,20) \rput(15,25){$2$}

\endpspicture
}
-
\raisebox{-40pt}{
\pspicture(0,0)(30,40)

\blankbox(0,30) \rput(5,35){$2$}
\blankbox(0,20) \rput(5,25){$3$}

\emptygraybox(10,30)

\blankbox(10,20) \rput(15,25){$1$}

\endpspicture
}.
\end{equation}

More generally, arbitrary minors of $Z$ can be expressed as maximal minors of an augmented matrix $\hat{Z} = \left( Z \,|\, I \right)$, where $I$ is the $n\times n$ identity matrix.  To specify columns of $\hat{Z}$, we introduce a second alphabet $\{\textcolor{blue}{\hat{1}},\textcolor{blue}{\hat{2}},\ldots\}$, so $\textcolor{blue}{\hat\imath}$ indicates the $i$th column of the identity matrix (i.e., the $(n+i)$th column of $\hat{Z}$).

The classical relations for maximal minors imply corresponding relations for minors on any given set of rows.  (See, e.g., \cite[\S6]{bcrv}.)  For example, the minor $\Delta_{\{2,3,5\},\{1,2,5\}}(Z)$ is equal to $\Delta_{\{1,\ldots,5\},\{\textcolor{blue}{\hat{1}},1,2,\textcolor{blue}{\hat{4}},5\}}(\hat{Z})$, and this can be represented as
\begin{equation*}
\raisebox{-35pt}{
\pspicture(0,0)(20,50)

\emptygraybox(10,40)
\emptygraybox(10,10)

\blankbox(10,30) \rput(15,35){$1$}
\blankbox(10,20) \rput(15,25){$2$}
\blankbox(10,0)  \rput(15,5){$5$}

\endpspicture
}
=
\raisebox{-35pt}{
\pspicture(10,0)(20,50)

\emptygraybox(10,40) \rput(15,45){\small{$\textcolor{blue}{\hat{1}}$}}
\emptygraybox(10,10) \rput(15,15){\small{$\textcolor{blue}{\hat{4}}$}}

\blankbox(10,30) \rput(15,35){$1$}
\blankbox(10,20) \rput(15,25){$2$}
\blankbox(10,0)  \rput(15,5){$5$}

\endpspicture
}\;,
\end{equation*}
corresponding to
\[
\left|\begin{array}{ccc}
z_{21} & z_{22} & z_{25} \\
z_{31} & z_{32} & z_{35} \\
z_{51} & z_{52} & z_{55}
\end{array}\right|
=
\left|\begin{array}{ccccc}
1 & z_{11} & z_{12} & 0 & z_{15} \\
0 & z_{21} & z_{22} & 0 & z_{25} \\
0 & z_{31} & z_{32} & 0 & z_{35} \\
0 & z_{41} & z_{42} & 1 & z_{45} \\
0 & z_{51} & z_{52} & 0 & z_{55}
\end{array}\right|.
\]

For example, we have a relation
\begin{equation}\label{e.2column-relation}
\begin{aligned}
\raisebox{-35pt}{
\pspicture(0,0)(20,50)

\emptygraybox(0,40)
\emptygraybox(0,10)
\emptygraybox(0,0)

\blankbox(0,30) \rput(5,35){$1$}
\blankbox(0,20) \rput(5,25){$2$}


\emptygraybox(10,40)
\emptygraybox(10,10)

\blankbox(10,30) \rput(15,35){$1$}
\blankbox(10,20) \rput(15,25){$\bm{3}$}
\blankbox(10,0)  \rput(15,5){$5$}

\endpspicture
}
=
\raisebox{-35pt}{
\pspicture(0,0)(20,50)

\emptygraybox(0,40) \rput(5,45){\small$\textcolor{blue}{\hat{1}}$}
\emptygraybox(0,10) \rput(5,15){\small$\textcolor{blue}{\hat{4}}$}
\emptygraybox(0,0) \rput(5,5){\small$\textcolor{blue}{\hat{5}}$}

\blankbox(0,30) \rput(5,35){$1$}
\blankbox(0,20) \rput(5,25){$2$}


\emptygraybox(10,40) \rput(15,45){\small$\textcolor{blue}{\hat{1}}$}
\emptygraybox(10,10) \rput(15,15){\small$\textcolor{blue}{\hat{4}}$}

\blankbox(10,30) \rput(15,35){$1$}
\blankbox(10,20) \rput(15,25){$\bm{3}$}
\blankbox(10,0)  \rput(15,5){$5$}

\endpspicture
}
&=
\raisebox{-35pt}{
\pspicture(0,0)(20,50)

\emptygraybox(0,10) \rput(5,15){\small$\textcolor{blue}{\hat{4}}$}
\emptygraybox(0,0) \rput(5,5){\small$\textcolor{blue}{\hat{5}}$}

\blankbox(0,40) \rput(5,45){$\bm{3}$}
\blankbox(0,30) \rput(5,35){$1$}
\blankbox(0,20) \rput(5,25){$2$}


\emptygraybox(10,40) \rput(15,45){\small$\textcolor{blue}{\hat{1}}$}
\emptygraybox(10,20) \rput(15,25){\small$\textcolor{blue}{\hat{1}}$}
\emptygraybox(10,10) \rput(15,15){\small$\textcolor{blue}{\hat{4}}$}

\blankbox(10,30) \rput(15,35){$1$}
\blankbox(10,0)  \rput(15,5){$5$}

\endpspicture
}
+
\raisebox{-35pt}{
\pspicture(0,0)(20,50)

\emptygraybox(0,40) \rput(5,45){\small$\textcolor{blue}{\hat{1}}$}
\emptygraybox(0,10) \rput(5,15){\small$\textcolor{blue}{\hat{4}}$}
\emptygraybox(0,0) \rput(5,5){\small$\textcolor{blue}{\hat{5}}$}

\blankbox(0,30) \rput(5,35){$\bm{3}$}
\blankbox(0,20) \rput(5,25){$2$}


\emptygraybox(10,40) \rput(15,45){\small$\textcolor{blue}{\hat{1}}$}
\emptygraybox(10,10) \rput(15,15){\small$\textcolor{blue}{\hat{4}}$}

\blankbox(10,30) \rput(15,35){$1$}
\blankbox(10,20) \rput(15,25){$1$}
\blankbox(10,0)  \rput(15,5){$5$}

\endpspicture
}
+
\raisebox{-35pt}{
\pspicture(0,0)(20,50)

\emptygraybox(0,40) \rput(5,45){\small$\textcolor{blue}{\hat{1}}$}
\emptygraybox(0,10) \rput(5,15){\small$\textcolor{blue}{\hat{4}}$}
\emptygraybox(0,0) \rput(5,5){\small$\textcolor{blue}{\hat{5}}$}

\blankbox(0,30) \rput(5,35){$1$}
\blankbox(0,20) \rput(5,25){$\bm{3}$}


\emptygraybox(10,40) \rput(15,45){\small$\textcolor{blue}{\hat{1}}$}
\emptygraybox(10,10) \rput(15,15){\small$\textcolor{blue}{\hat{4}}$}

\blankbox(10,30) \rput(15,35){$1$}
\blankbox(10,20) \rput(15,25){$2$}
\blankbox(10,0)  \rput(15,5){$5$}

\endpspicture
}
+
\raisebox{-35pt}{
\pspicture(0,0)(20,50)

\emptygraybox(0,40) \rput(5,45){\small$\textcolor{blue}{\hat{1}}$}
\emptygraybox(0,0) \rput(5,5){\small$\textcolor{blue}{\hat{5}}$}

\blankbox(0,30) \rput(5,35){$1$}
\blankbox(0,20) \rput(5,25){$2$}
\blankbox(0,10) \rput(5,15){${\bm{3}}$}


\emptygraybox(10,40) \rput(15,45){\small$\textcolor{blue}{\hat{1}}$}
\emptygraybox(10,20) \rput(15,25){\small$\textcolor{blue}{\hat{4}}$}
\emptygraybox(10,10) \rput(15,15){\small$\textcolor{blue}{\hat{4}}$}

\blankbox(10,30) \rput(15,35){$1$}
\blankbox(10,0)  \rput(15,5){$5$}

\endpspicture
}
+
\raisebox{-35pt}{
\pspicture(0,0)(20,50)

\emptygraybox(0,40) \rput(5,45){\small$\textcolor{blue}{\hat{1}}$}
\emptygraybox(0,10) \rput(5,15){\small$\textcolor{blue}{\hat{4}}$}

\blankbox(0,30) \rput(5,35){$1$}
\blankbox(0,20) \rput(5,25){$2$}
\blankbox(0,0) \rput(5,5){${\bm{3}}$}

\emptygraybox(10,40) \rput(15,45){\small$\textcolor{blue}{\hat{1}}$}
\emptygraybox(10,20) \rput(15,25){\small$\textcolor{blue}{\hat{5}}$}
\emptygraybox(10,10) \rput(15,15){\small$\textcolor{blue}{\hat{4}}$}

\blankbox(10,30) \rput(15,35){$1$}
\blankbox(10,0)  \rput(15,5){$5$}

\endpspicture
}
\\
&=
%
\quad\; 0\quad + \quad\;0\quad +
\raisebox{-35pt}{
\pspicture(0,0)(20,50)

\emptygraybox(0,40) 
\emptygraybox(0,10) 
\emptygraybox(0,0) 

\blankbox(0,30) \rput(5,35){$1$}
\blankbox(0,20) \rput(5,25){$\bm{3}$}


\emptygraybox(10,40) 
\emptygraybox(10,10) 

\blankbox(10,30) \rput(15,35){$1$}
\blankbox(10,20) \rput(15,25){$2$}
\blankbox(10,0)  \rput(15,5){$5$}

\endpspicture
}
+ \quad0\quad -
\raisebox{-35pt}{
\pspicture(0,0)(20,50)

\emptygraybox(0,40) 
\emptygraybox(0,10) 

\blankbox(0,30) \rput(5,35){$1$}
\blankbox(0,20) \rput(5,25){$2$}
\blankbox(0,0) \rput(5,5){${\bm{3}}$}

\emptygraybox(10,40) 
\emptygraybox(10,10) 
\emptygraybox(10,0)  

\blankbox(10,30) \rput(15,35){$1$}
\blankbox(10,20) \rput(15,25){$5$}

\endpspicture
}\,,
\end{aligned}
\end{equation}
that is,
\begin{align*}
  &\Delta_{\{2,3\},\{1,2\}}(Z)\cdot \Delta_{\{2,3,5\},\{1,3,5\}}(Z) \\
 &\qquad =  \Delta_{\{2,3\},\{1,3\}}(Z) \cdot  \Delta_{\{2,3,5\},\{1,2,5\}}(Z) -  \Delta_{\{2,3,5\},\{1,2,3\}}(Z) \cdot  \Delta_{\{2,3\},\{1,5\}}(Z),
\end{align*}
coming from the indicated relation for the maximal minors of $\hat{Z}$.  (In this example, the columns were nested, but the same construction evidently applies to any pair of columns $\ba$, $\ba'$.)

\begin{remark}
When $D$ consists of a nested sequence of columns, so $\ba^{(1)} \supseteq \ba^{(2)} \supseteq \cdots \supseteq \ba^{(m)}$, these two-column relations generate all the $D$-exchange relations.  This is one of the fundamental theorems of classical invariant theory: the quadratic Pl\"ucker relations generate all relations among maximal minors of a generic matrix.  (We may restrict to the submatrix of $Z$ on rows $\ba^{(1)}$, and re-order rows so all $\ba^{(j)}$ are top-justified.  Missing boxes in a column $\ba^{(j)}$ may then be considered as filled by very large $\textcolor{blue}{\hat\imath}$; see \cite[\S8.1]{fulton-yt} or \cite[\S6]{bcrv} for further details.)
\end{remark}

\begin{remark}\label{r.cubic}
On the other hand, for general diagrams $D$, the exchange relations are poorly understood---not all relations come from quadratic ones.  For instance, the cubic relation
\begin{equation}
\raisebox{-20pt}{
\pspicture(0,0)(30,30)

\emptygraybox(0,0)

\blankbox(0,20) \rput(5,25){$1$}
\blankbox(0,10) \rput(5,15){$2$}


\emptygraybox(10,10)

\blankbox(10,20)  \rput(15,25){$1$}
\blankbox(10,0)  \rput(15,5){$3$}

\emptygraybox(20,20)

\blankbox(20,10) \rput(25,15){$1$}
\blankbox(20,0) \rput(25,5){$4$}

\endpspicture
}
+
\raisebox{-20pt}{
\pspicture(0,0)(30,30)

\emptygraybox(0,0)

\blankbox(0,20) \rput(5,25){$1$}
\blankbox(0,10) \rput(5,15){$3$}


\emptygraybox(10,10)

\blankbox(10,20)  \rput(15,25){$1$}
\blankbox(10,0)  \rput(15,5){$4$}

\emptygraybox(20,20)

\blankbox(20,10) \rput(25,15){$1$}
\blankbox(20,0) \rput(25,5){$2$}

\endpspicture
}
+
\raisebox{-20pt}{
\pspicture(0,0)(30,30)

\emptygraybox(0,0)

\blankbox(0,20) \rput(5,25){$1$}
\blankbox(0,10) \rput(5,15){$4$}


\emptygraybox(10,10)

\blankbox(10,20)  \rput(15,25){$1$}
\blankbox(10,0)  \rput(15,5){$2$}

\emptygraybox(20,20)

\blankbox(20,10) \rput(25,15){$1$}
\blankbox(20,0) \rput(25,5){$3$}

\endpspicture
}
-
\raisebox{-20pt}{
\pspicture(0,0)(30,30)

\emptygraybox(0,0)

\blankbox(0,20) \rput(5,25){$1$}
\blankbox(0,10) \rput(5,15){$4$}


\emptygraybox(10,10)

\blankbox(10,20)  \rput(15,25){$1$}
\blankbox(10,0)  \rput(15,5){$3$}

\emptygraybox(20,20)

\blankbox(20,10) \rput(25,15){$1$}
\blankbox(20,0) \rput(25,5){$2$}

\endpspicture
}
-
\raisebox{-20pt}{
\pspicture(0,0)(30,30)

\emptygraybox(0,0)
 
\blankbox(0,20) \rput(5,25){$1$}
\blankbox(0,10) \rput(5,15){$3$}


\emptygraybox(10,10)

\blankbox(10,20)  \rput(15,25){$1$}
\blankbox(10,0)  \rput(15,5){$2$}

\emptygraybox(20,20)

\blankbox(20,10) \rput(25,15){$1$}
\blankbox(20,0) \rput(25,5){$4$}

\endpspicture
}
-
\raisebox{-20pt}{
\pspicture(0,0)(30,30)

\emptygraybox(0,0)
 
\blankbox(0,20) \rput(5,25){$1$}
\blankbox(0,10) \rput(5,15){$2$}


\emptygraybox(10,10)

\blankbox(10,20)  \rput(15,25){$1$}
\blankbox(10,0)  \rput(15,5){$4$}

\emptygraybox(20,20)

\blankbox(20,10) \rput(25,15){$1$}
\blankbox(20,0) \rput(25,5){$3$}

\endpspicture
}
=0
\end{equation}
is minimal, i.e., not a consequence of other relations.  See \cite[\S6.5, (6.8)]{bcrv}.
\end{remark}

Now we can define the $\kk$-module $E^D$.

\begin{definition}
Given a finitely generated $\kk$-module $E$ and a diagram $D$, the {\it Schur module} $E^D$ is defined by the universal property
\[
  \Hom( E^D, M )  = \left\{ \phi\colon E^{\times D} \to M \,\left|\, \begin{array}{l} \text{(1) } \phi \text{ is multilinear;} \\ \text{(2) }\phi \text{ is alternating on columns; and} \\ \text{(3) } \phi \text{ satisfies the }D\text{-exchange relations}\end{array} \right.\right\}
\]
for any $\kk$-module $M$.
\end{definition}

We have constructed $E^D$ in the case where $E$ is free, as the submodule of $\kk[Z]$ spanned by certain products of determinants.  In general, let $F\tto E$ be a presentation, so $F$ is a finitely generated free module.  We have surjections $TF^D \tto TE^D$ and $TF^D \tto F^D$, so we may construct $E^D$ as the pushout, i.e., the universal (source) module fitting into the diagram
\begin{equation}\label{e.pushout-schur}
\begin{tikzcd}
  TF^D \ar[r,two heads] \ar[d,two heads] & F^D \ar[d] \\
  TE^D \ar[r] & E^D.
\end{tikzcd}
\end{equation}
(The homomorphisms $F^D \to E^D$ and $TE^D \to E^D$ are also surjective, by general properties of pushouts.)  The module $E^D$ thus constructed satisfies the universal property of the definition: dualizing the pushout, $\Hom(E^D,M)$ is a pullback, which here means that
\[
  \Hom(E^D,M) = \Hom(F^D,M) \cap \Hom(TE^D,M)
\]
as submodules of $\Hom(TF^D,M)$.

From the definitions, $E^D$ only depends on $D$ up to permutations of the rows and columns.  That is, if $D'$ is obtained from $D$ by permuting some rows and columns, there is a canonical isomorphism $E^D \isom E^{D'}$.

\begin{example}
There are no exchange relations across disjoint sets $\ba^{(1)}$ and $\ba^{(2)}$, because the corresponding determinants are taken in disjoint sets of variables.  It follows that if $D=D'\cup D''$, with the columns in $D'$ occupying disjoint rows from those in $D''$, then $E^D = E^{D'} \otimes E^{D''}$.
\end{example}

\begin{example}
If $D=\lambda$ is a Young diagram---that is, it consists of $\lambda_i$ boxes in row $i$, with $\lambda_1\geq \cdots \geq \lambda_n>0$---then $E^D = E^\lambda$ is the classical Schur module: the construction is identical to the one given in \cite[\S8.1]{fulton-yt}.  So the same is true whenever the rows and columns of $D$ can be permuted to form a Young diagram.
\end{example}

\begin{example}
When $\kk$ is a field of characteristic zero, so $E=\kk^n$ is a vector space, $E^D$ agrees with the module studied by Reiner and Shimozono \cite{rs}.  With some further restrictions on the diagram, it also agrees with the module introduced by Kra\'skiewicz and Pragacz \cite{kp}, which in turn is modelled on the construction studied in detail by Akin, Buchsbaum, and Weyman in the case $D=\lambda$ \cite{abw}.  
\end{example}

\subsection{Group actions}

Any homomorphism $E\to F$ determines a homomorphism $E^D \to F^D$.  In particular, $E^D$ is a $GL(E)$-module.  In the case where $E$ is free, so $E^D$ embeds in $\kk[Z]$, this action is induced from the action on {\it entries} of the matrix $Z=(z_{ij})$, by
\[
  g\cdot (z_{ij}) = (g\cdot z_{ij}).
\]
Recalling that $z_{ij} = e_i \otimes z_j$, so $g\cdot z_{ij} = e_i\otimes g\cdot z_j$, one sees that the span of products of determinants on rows specified by $D$, and all possible choices of columns, is preserved by this action.

When $\kk$ is a field, and a maximal torus $T\subset GL(E)$ is chosen, the \define{character} of a $GL(E)$-module $M$ is the trace of the action restricted to $T$.  This is an element of the representation ring of $T$, written
\[
  \tchar_T(M) \in R(T).
\]
We will generally fix a basis so that $T$ consists of diagonal matrices, and for an element $t=\diag(t_1,\ldots,t_n)$, the character $x_i\colon T \to \kk^*$ is defined by $x_i(t) = t_i$.  This identifies $R(T)$ with $\Z[x_1^{\pm},\ldots,x_n^{\pm}]$, so $\tchar_T(M)$ is a Laurent polynomial.  We have
\[
 \tchar_T(M) = \sum_{\lambda\in\Z^n} (\dim M_\lambda)\, x^\lambda,
\]
where $x^\lambda = x_1^{\lambda_1}\cdots x_n^{\lambda_n}$, and $M_\lambda$ is the subspace of $M$ where $T$ acts with weight $\lambda$, meaning $t\cdot v = t_1^{\lambda_1}\cdots t_n^{\lambda_n} v$ for all $v\in M_\lambda$.  (For all the modules we consider, $\tchar_T(M)$ will in fact be a polynomial.)

In the following examples, we assume $\kk$ is a field.

\begin{example}
If $D$ consists of a single box, then $E^D = E$ with its standard action of $GL(E)$.  Concretely, in a basis $z_1,\ldots,z_n$, suppose $t = \diag(t_1,t_2,\ldots,t_n)$, so that $t\cdot z_j = t_jz_j$.  Then for any row $i$, $t$ acts on the row vector of $Z$ by
\[
  t\cdot \big( z_{i,1} \,\big|\, z_{i,2} \,\big|\, \cdots \,\big|\, z_{i,n} \big) = \big( t_1 z_{i,1} \,\big|\, t_2 z_{i,2} \,\big|\,\ \cdots \,\big|\, t_n z_{i,n} \big).
\]
The character of such a module is $\tchar_T(E) = x_1+\cdots+x_n$.
\end{example}

\begin{example}\label{ex.13}
Suppose $D = \big[\ba\big] = \big[ \{1,3\} \big]$ and $E$ is $3$-dimensional, with basis $z_1,z_2,z_3$.  Then $E^D = E^{\ba} = \exterior^2 E$.  We get a basis $\{z_{\ba,\bb}\}$ for $E^{\ba}$, for $\bb=\{1,2\},\{1,3\},\{2,3\}$.  These map to the determinants
\[
  \Delta_{\{1,3\},\{1,2\}} = \left|\begin{array}{cc} z_{11} & z_{12} \\ z_{31} & z_{32} \end{array}\right|, \; 
  \Delta_{\{1,3\},\{1,3\}} = \left|\begin{array}{cc} z_{11} & z_{13} \\ z_{31} & z_{33} \end{array}\right|, \; 
  \Delta_{\{1,3\},\{2,3\}} = \left|\begin{array}{cc} z_{12} & z_{13} \\ z_{32} & z_{33} \end{array}\right|.
\]
For a diagonal element $t=\diag(t_1,t_2,t_3)$, we have
\begin{align*}
  t\cdot \Delta_{\{1,3\},\{1,2\}} &= t_1 t_2 \Delta_{\{1,3\},\{1,2\}}, \\
  t\cdot \Delta_{\{1,3\},\{1,3\}} &= t_1 t_3 \Delta_{\{1,3\},\{1,3\}},\\
  t\cdot \Delta_{\{1,3\},\{2,3\}} &= t_2 t_3 \Delta_{\{1,3\},\{2,3\}}.
\end{align*}
So the character of this module is $\tchar_T(E^{\ba}) = x_1 x_2 + x_1 x_3 + x_2 x_3$.

The lower-triangular element
\[
  g = \left( \begin{array}{ccc} 1 & 0 & 0 \\ a & 1 & 0 \\ 0 & a & 1 \end{array}\right)
\]
acts by $g\cdot z_j = z_j + a z_{j+1}$ for $j=1,2$, and $g\cdot z_3 = z_3$.  So
\begin{align*}
  g\cdot \Delta_{\{1,3\},\{1,2\}} &= \Delta_{\{1,3\},\{1,2\}} + a\Delta_{\{1,3\},\{1,3\}} + a^2\Delta_{\{1,3\},\{2,3\}}, \\
  g\cdot \Delta_{\{1,3\},\{1,3\}} &= \Delta_{\{1,3\},\{1,3\}} + a\Delta_{\{1,3\},\{2,3\}},\\
  g\cdot \Delta_{\{1,3\},\{2,3\}} &= \Delta_{\{1,3\},\{2,3\}}.
\end{align*}
Similar calculations show that $E^D$ is generated by $\Delta_{\{1,3\},\{1,2\}}$ as a $B$-module, where $B$ is the subgroup of lower-triangular elements.
\end{example}


\begin{remark}
The Schur modules considered here are dual to the {\it Weyl modules} often seen in the literature.  Some authors consider the vector space of $n\times n$ matrices $Y=(y_{ij})$, with $GL_n$ acting on the ring of polynomial {\em functions} $\Sym^*(Y^\dual)$ by $(g\cdot f)(Y) = f(g^{-1}Y)$, and define a Weyl module $\mathcal{M}_D$ to be the subspace spanned by certain products of determinants.  (This works over a field of characteristic zero; more canonically one should take $\mathcal{M}_D$ to be a quotient of the divided power algebra.)  To recover Schur and Schubert polynomials, they study the {\it dual characters} of such Weyl modules.  (See, e.g., \cite{magyar, mst} for the characteristic zero version, and \cite{abw} for the more general story.)  The dual of a Weyl module is the double-dual of our Schur module, so their characters agree.
\end{remark}

\section{Flagged Schur modules}\label{s.flagged}

Let $E$ be equipped with a quotient flag:
\[
  E_\bull: E = E_n \tto \cdots \tto E_1 \tto 0.
\]
A morphism of flags $E_\bull \to F_\bull$ is just a collection of $\kk$-linear maps $E_i \to F_i$ which commute with the quotient projections.

\subsection{Constructions and definitions}

The constructions of the previous section extend directly to flagged modules.  Before giving the general version, we quickly go through the flagged analogues of the intermediate steps, $E^\ba_\bull$ and $TE^D_\bull$.

First, we can define $E^\ba_\bull$ via an analogous universal property.  To describe it, for each $p$, let $K_p = \ker(E \tto E_p)$.  Given $\ba = (a_1<\cdots<a_r)$, we say a multilinear function $\phi\colon E^{\times r} \to M$ is {\it flagged} with respect to $\ba$ if
\begin{equation}\label{e.flagged}
  \phi(v_1,\ldots,v_r) = 0 \text{ whenever } r-i \text{ of } \{v_1,\ldots,v_r\} \text{ lie in } K_{a_{i+1}},
\end{equation}
for some $0\leq i<r$.

\begin{definition}\label{d.flagged-column}
The $\kk$-module $E^\ba_\bull$ is defined by the universal property
\[
  \Hom(E^{\ba}_\bull,M) = \left\{ E^{\times r} \xrightarrow{\phi} M \,\left|\, \begin{array}{l} \phi \text{ is multilinear, alternating, and} \\ \text{flagged with respect to }\ba 
  \end{array} \right. \right\}
\]
for any $\kk$-module $M$.
\end{definition}

\noindent 
That is, $E^\ba_\bull$ is the universal source for multilinear, alternating, flagged maps out of $E^{\times r}$.

This universal property constructs the module $E^\ba_\bull$ as the quotient $E^\ba/K$, where $K\subseteq E^\ba$ is the image of the canonical homomorphism
\[
 \bigoplus_{i=0}^{r-1} \exterior^iE \otimes \exterior^{r-i} K_{a_{i+1}} \to \exterior^r E.
\]

We say a flag $E_\bull$ is \define{free} if $E$ and each $E_i$ are free modules.  Suppose $E_\bull$ is free and $\rk E_i = i$.  We can choose bases $z_{i,j}$ for $E_i$ (with $1\leq j\leq i$), so that the projection $E_i \to E_{i-1}$ is given by
\[
  z_{i,j} \mapsto \begin{cases} z_{i-1,j} & \text{if }j\leq {i-1}; \\ 0 & \text{if }j=i. \end{cases}
\]
When such bases are chosen, we say $E_\bull$ is a \define{standard flag}.\footnote{We will continue to assume $\rk E_i = i$ throughout the article.  For remarks on the more general case where $\rk E_i = d_i$, see Appendix~\ref{app.parabolic}.}

Given a subset $\bb = \{b_1<b_2< \cdots < b_r \}$, let us write
\[
  \bb \leq \ba \quad \text{ if } \quad b_i \leq a_i \text{ for all }i.
\]
Then $E^\ba_\bull$ has a basis of elements
\[
  z_{\ba,\bb} := z_{a_1,b_1} \wedge \cdots \wedge z_{a_r,b_r},
\]
ranging over subsets $\bb \leq \ba$.  (Our conventions make $z_{\ba,\bb}=0$ if $\bb\not\leq\ba$.)

As before, we will indicate basis elements by labelling diagrams, now requiring that labels in row $i$ be no larger than $i$.

\begin{example}\label{ex.235}
For $\ba=\{2,3,5\}$, the module $E^\ba_\bull$ has a basis consisting of the following seven elements:
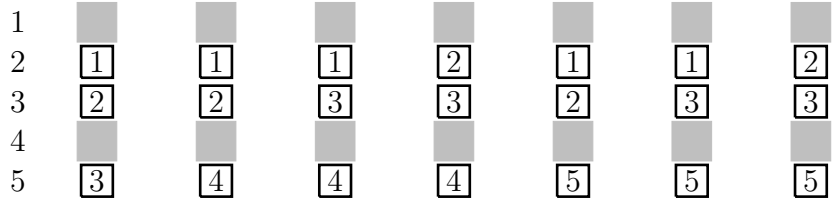
\begin{center}
\pspicture(0,-5)(30,55)

\rput(-5,45){$1$}
\rput(-5,35){$2$}
\rput(-5,25){$3$}
\rput(-5,15){$4$}
\rput(-5,5){$5$}

\emptygraybox(10,40)
\emptygraybox(10,10)

\blankbox(10,30) \rput(15,35){$1$}
\blankbox(10,20) \rput(15,25){$2$}
\blankbox(10,0) \rput(15,5){$3$}

\endpspicture
\pspicture(0,-5)(30,55)

\emptygraybox(10,40)
\emptygraybox(10,10)

\blankbox(10,30) \rput(15,35){$1$}
\blankbox(10,20) \rput(15,25){$2$}
\blankbox(10,0) \rput(15,5){$4$}

\endpspicture
\pspicture(0,-5)(30,55)

\emptygraybox(10,40)
\emptygraybox(10,10)

\blankbox(10,30) \rput(15,35){$1$}
\blankbox(10,20) \rput(15,25){$3$}
\blankbox(10,0) \rput(15,5){$4$}

\endpspicture
\pspicture(0,-5)(30,55)

\emptygraybox(10,40)
\emptygraybox(10,10)

\blankbox(10,30) \rput(15,35){$2$}
\blankbox(10,20) \rput(15,25){$3$}
\blankbox(10,0) \rput(15,5){$4$}

\endpspicture
\pspicture(0,-5)(30,55)

\emptygraybox(10,40)
\emptygraybox(10,10)

\blankbox(10,30) \rput(15,35){$1$}
\blankbox(10,20) \rput(15,25){$2$}
\blankbox(10,0) \rput(15,5){$5$}

\endpspicture
\pspicture(0,-5)(30,55)

\emptygraybox(10,40)
\emptygraybox(10,10)

\blankbox(10,30) \rput(15,35){$1$}
\blankbox(10,20) \rput(15,25){$3$}
\blankbox(10,0) \rput(15,5){$5$}

\endpspicture
\pspicture(0,-5)(30,55)

\emptygraybox(10,40)
\emptygraybox(10,10)

\blankbox(10,30) \rput(15,35){$2$}
\blankbox(10,20) \rput(15,25){$3$}
\blankbox(10,0) \rput(15,5){$5$}

\endpspicture

\end{center}
\end{example}

Given a diagram $D = \{\ba^{(1)},\ldots,\ba^{(m)}\}$, the flagged tensor product module $TE^D_\bull$ is defined similarly to $TE^D$, by
\[
  TE^D_\bull = E^{\ba^{(1)}}_\bull \otimes E^{\ba^{(2)}}_\bull \otimes \cdots \otimes E^{\ba^{(m)}}_\bull.
\]
This has a basis consisting of tensors $z_\tau$, also defined as before, but now $\tau$ ranges over {\it column-strict flagged fillings}, meaning an entry in row $i$ can be at most $i$.

\begin{definition}
Given a finitely generated flagged $\kk$-module $E_\bull$ and a diagram $D$, the {\it flagged Schur module} $E_\bull^D$ is defined by the universal property
\[
  \Hom( E_\bull^D, M )  = \left\{ E^{\times D} \xrightarrow{\phi} M \,\left|\, \begin{array}{l} \text{(1) } \phi \text{ is multilinear;} \\ \text{(2) }\phi \text{ is alternating and flagged on each column; and} \\ 
  \text{(3) } \phi \text{ satisfies the }D\text{-exchange relations}\end{array} \right.\right\}
\]
for any $\kk$-module $M$.  (Here ``flagged'' refers to the condition \eqref{e.flagged} from Definition~\ref{d.flagged-column}, and the $D$-exchange relations are defined in Definition~\ref{d.D-exchange}.)
\end{definition}

As will be spelled out below, one should regard $(\cdot)^D$ as being functorial in the flagged module $E_\bull$.  Logically one might write $(E_\bull)^D$, but we generally suppress the parentheses.

Analogously to our construction of $E^D$ as a quotient of $TE^D$ in \S\ref{s.schur}, the universal property constructs $E^D_\bull$ as a quotient of $TE_\bull^D$.  In fact, we have another pushout diagram
\begin{equation}\label{e.pushout-flag}
\begin{tikzcd}
   TE^D \ar[r,two heads] \ar[d,two heads] & E^D \ar[d, two heads] \\
   TE^D_\bull \ar[r,two heads]  & E_\bull^D.
\end{tikzcd}
\end{equation}

The $D$-exchange relations for flagged modules can also be obtained as the kernel of a map to a polynomial ring, just as we defined them for $E^D$.  Let
\[
  \Sym^*(E_\bull) = \Sym^*(E_1)\otimes \Sym^*(E_2)\otimes \cdots\otimes \Sym^*(E_n).
\]
In case $E_\bull$ is free and standard, with compatible bases $z_{ij}$ for $E_i$ as above, this is a polynomial ring in variables
\begin{align*}
&  z_{11},  \\
&  z_{21}, z_{22}, \\
&  \vdots \\
&  z_{n1}, z_{n2}, \ldots, z_{n,n}.
\end{align*}
Let $Z_\bull = (z_{ij})_{i\geq j}$ be the $n\times n$  lower-triangular matrix of the form suggested above, so $\Sym^*(E_\bull) \isom \kk[Z_\bull]$.  We have a homomorphism
\begin{align*}
  TE^D_\bull &\to \kk[Z_\bull] \\
   z_\tau &\to \Delta_\tau,
\end{align*}
completing the diagram
\[
\begin{tikzcd}
  TE^D \ar[r, two heads] \ar[d,two heads] & E^D \ar[r,hook] \ar[d,two heads] & \kk[Z] \ar[d, two heads] \\
  TE_\bull^D \ar[r, two heads]  & E_\bull^D \ar[r,hook]  & \kk[Z_\bull].
\end{tikzcd}
\]
So in the free case, the homomorphism $E^D \tto E_\bull^D$ is induced by the map $\kk[Z] \tto \kk[Z_\bull]$ which sends $z_{ij} \mapsto 0$ for $j>i$.

\subsection{Group actions}

A homomorphism of flags $E_\bull \to F_\bull$ determines a homomorphism of flagged modules $E_\bull^D \to F_\bull^D$.  In particular, $E_\bull^D$ is a representation of the Borel subgroup $B\subset GL(E)$ which preserves $E_\bull$.

The construction of $\Sym^*(E_\bull)$ is also functorial, so a homomorphism $E_\bull \to F_\bull$ determines a homomorphism $\Sym^*(E_\bull) \to \Sym^*(F_\bull)$.  When $E_\bull$ and $F_\bull$ are free, the homomorphism $E_\bull^D \to F_\bull^D$ is induced by this homomorphism of polynomial rings.  In particular, the action of $g\in B$ on $E_\bull^D$ is induced from the action on entries of $Z_\bull = (z_{ij})_{i\geq j}$, just as before.  
Our choice of basis identifies $B$ with the Borel group of lower-triangular matrices.  

Assuming $\kk$ is a field and a maximal torus $T\subset B$ is chosen, we will write
\[
 \fS_D = \tchar_T(E_\bull^D)
\]
for the character.  Our conventions make $\fS_D$ a polynomial in $\Z[x_1,\ldots,x_n]$, not merely a Laurent polynomial.

In the following examples, we assume $\kk$ is a field.

\begin{example}\label{ex.113}
Building on Example~\ref{ex.13}, suppose $D = \big[ \{1\},\; \{1,3\} \big]$ and $E_\bull$ is a standard flag.  Then $TE_\bull^D = E_\bull^{\{1\}}\otimes E_\bull^{\{1,3\}}$, and this has a basis consisting of the two elements $z_{\{1\},\{1\}}\otimes z_{\{1,3\},\{1,2\}}$ and $z_{\{1\},\{1\}}\otimes z_{\{1,3\},\{1,3\}}$, corresponding to the two column-strict flagged fillings of $D$:

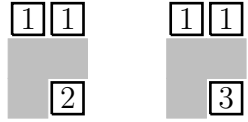
\begin{center}
\pspicture(0,0)(30,40)

\emptygraybox(0,10)
\emptygraybox(0,0)

\blankbox(0,20) \rput(5,25){$1$}


\emptygraybox(10,10)

\blankbox(10,20)  \rput(15,25){$1$}
\blankbox(10,0)  \rput(15,5){$2$}

\endpspicture
\pspicture(-10,0)(30,40)

\emptygraybox(0,10)
\emptygraybox(0,0)

\blankbox(0,20) \rput(5,25){$1$}


\emptygraybox(10,10)

\blankbox(10,20)  \rput(15,25){$1$}
\blankbox(10,0)  \rput(15,5){$3$}

\endpspicture

\end{center}

\noindent
These map to the polynomials
\[
  z_{11} \Delta_{\{1,3\},\{1,2\}} = z_{11} \left|\begin{array}{cc} z_{11} & 0 \\ z_{31} & z_{32} \end{array}\right|, \quad 
  z_{11} \Delta_{\{1,3\},\{1,3\}} = z_{11} \left|\begin{array}{cc} z_{11} & 0 \\ z_{31} & z_{33} \end{array}\right|.
\]
For a diagonal element $t=\diag(t_1,t_2,t_3)$, we have
\begin{align*}
  t\cdot \left( z_{11} \Delta_{\{1,3\},\{1,2\}} \right)  &= t_1^2 t_2 \left( z_{11} \Delta_{\{1,3\},\{1,2\}}\right), \\
  t\cdot \left( z_{11} \Delta_{\{1,3\},\{1,3\}}\right) &= t_1^2 t_3 \left(z_{11} \Delta_{\{1,3\},\{1,3\}}\right).
\end{align*}
Evidently there are no relations, so the map $TE_\bull^D \tto E_\bull^D$ is an isomorphism, and the character of this module is
\[
\fS_D = x_1^2 x_2 + x_1^2 x_3.
\]  

The lower-triangular element
\[
  g = \left( \begin{array}{ccc} 1 & 0 & 0 \\ a & 1 & 0 \\ 0 & a & 1 \end{array}\right)
\]
acts by
\begin{align*}
  g\cdot \left( z_{11} \Delta_{\{1,3\},\{1,2\}} \right) &= z_{11} \left(\Delta_{\{1,3\},\{1,2\}} + a\Delta_{\{1,3\},\{1,3\}} \right), \\
  g\cdot \left( z_{11} \Delta_{\{1,3\},\{1,3\}}\right)  &= z_{11} \Delta_{\{1,3\},\{1,3\}}.
\end{align*}
This shows that $E_\bull^D$ is generated by $z_{11} \Delta_{\{1,3\},\{1,2\}}$ as a $B$-module.
\end{example}

\begin{example}\label{ex.13542}
Consider $D = \big[ \{2,3,4\},\; \{3\} \big]$.  Then $TE_\bull^D = E_\bull^{\{2,3,4\}} \otimes E_\bull^{\{3\}}$ is isomorphic to $\exterior^3 E_4 \otimes E_3$, so it is $12$-dimensional.  The module $E_\bull^D$ is a quotient of $TE_\bull^D$, and there is one relation:
\begin{center}
\pspicture(0,0)(20,40)

\emptygraybox(0,30)

\blankbox(0,20) \rput(5,25){$1$}
\blankbox(0,10) \rput(5,15){$3$}
\blankbox(0,0)  \rput(5,5){$4$}


\emptygraybox(10,30)
\emptygraybox(10,20)
\emptygraybox(10,0)

\blankbox(10,10)  \rput(15,15){$2$}

\endpspicture
\pspicture(-10,0)(20,40)

\rput(-5,20){$=$}

\emptygraybox(0,30)

\blankbox(0,20) \rput(5,25){$2$}
\blankbox(0,10) \rput(5,15){$3$}
\blankbox(0,0) \rput(5,5){$4$}

\emptygraybox(10,30)
\emptygraybox(10,20)
\emptygraybox(10,0)

\blankbox(10,10) \rput(15,15){$1$}

\endpspicture
\pspicture(-10,0)(30,40)

\rput(-5,20){$+$}

\emptygraybox(0,30)

\blankbox(0,20) \rput(5,25){$1$}
\blankbox(0,10) \rput(5,15){$2$}
\blankbox(0,0) \rput(5,5){$4$}

\emptygraybox(10,30)
\emptygraybox(10,20)
\emptygraybox(10,0)

\blankbox(10,10) \rput(15,15){$3$}

\rput(25,20){.}

\endpspicture
\end{center}
So $E_\bull^D$ is $11$-dimensional, with character
\begin{align*}
 \fS_D &= (x_1 x_2 x_3 + x_1 x_2 x_4 + x_1 x_3 x_4 + x_2 x_3 x_4)( x_1 + x_2 + x_3 ) - x_1 x_2 x_3 x_4 \\
  &= x_1^2 x_2 x_3 + x_1^2 x_2 x_4 + x_1^2 x_3 x_4 + x_1 x_2^2 x_3 + x_1 x_2^2 x_4 + x_1 x_2 x_3^2 \\
  &\quad + 2 x_1 x_2 x_3 x_4 + x_1 x_3^2 x_4 + x_2^2 x_3 x_4 + x_2 x_3^2 x_4.
\end{align*}
\end{example}

\begin{remark}
Quotient flags provide the natural setting for dealing with general modules.  However, so long as one works over a field of characteristic zero, the constructions in this section (and the rest of this paper) may be carried out using flags of subspaces instead, simply by dualizing.
\end{remark}

\section{Bergeron-Sottile operators}\label{s.Rk}

\subsection{Operators on indices and flags}

The \define{Bergeron-Sottile operator} $\Rop_k$ is defined on integer indices by
\begin{equation}
  \Rop_k(i) = \begin{cases} i &\text{if }i<k; \\ i-1 &\text{if }i\geq k. \end{cases}
\end{equation}
These satisfy the relations $\Rop_k^2 = \Rop_k \Rop_{k+1}$.  (They do not commute: for instance, $\Rop_2^2(3) = \Rop_2\Rop_3(3) = 1$, but $\Rop_3 \Rop_2(3) = 2$.)  These operators were defined and studied by Nadeau, Spink and Tewari via their action on the polynomial ring \cite{nst,nst2}.  To translate to their notation, let $\varpi_i = x_1 + x_2 + \cdots + x_i$, and define an endomorphism of the polynomial ring by $\Rop_k(\varpi_i) = \varpi_{\Rop_k(i)}$.  Written in terms of the $x$ variables, $\Rop_k$ acts on polynomials by
\[
  (\Rop_k f)(x_1,x_2,\ldots) = f(x_1,\ldots,x_{k-1},0,x_{k},\ldots).
\]
That is, $\Rop_k$ acts on the polynomial ring by
\[
 \Rop_k(x_i)=\begin{cases} x_i &\text{if }i<k, \\ 0 &\text{if }i=k, \\ x_{i-1} &\text{if } i>k. \end{cases}
\]

Given a flagged module $E_\bull$, for any integer $k$, let $\Rop_kE_\bull$ be the flagged module defined by $(\Rop_k E_\bull)_i = E_{\Rop_k(i)}$.  We often abbreviate the notation by writing $\Rop_k E_i$ for the $i$th part of this flag, so
\[
  \Rop_kE_i = \begin{cases} E_i & \text{if } i< k; \\ E_{i-1} & \text{if } i\geq k. \end{cases}
\]
The maps of the flag $\Rop_kE_\bull$ are the ones induced from $E_\bull$, but with $\Rop_k E_{k} \to \Rop_k E_{k-1}$ being the identity map on $E_{k-1}$.

We have surjections of flags $E_\bull = \Rop_{n+1}E_\bull \tto \Rop_nE_\bull \tto \cdots \tto \Rop_1E_\bull$, which expand as follows:

\begin{equation*}
\begin{tikzcd}
   &   & k+2  & k+1 & k & k-1 \\
\Rop_{k+1}E_\bull:  \ar[d,two heads, "\pi_k"]  & \cdots  \ar[r,two heads]  & E_{k+1} \ar[d,equals] \ar[r,two heads] &  E_k \ar[d, equals] \ar[r,equals] & E_k \ar[d,two heads] \ar[r,two heads] & E_{k-1}\ar[d, equals]  \ar[r,two heads] & \cdots \\
\Rop_kE_\bull:    & \cdots  \ar[r,two heads] & E_{k+1} \ar[r,two heads]&  E_{k} \ar[r,two heads] & E_{k-1} \ar[r,equals] & E_{k-1} \ar[r,two heads] & \cdots.
\end{tikzcd}
\end{equation*}

The algebra of operators $\Rop_k$ acts oppositely on flags: we have $\Rop_k^2E_\bull = \Rop_{k+1}\Rop_kE_\bull$.  For instance, $\Rop_2^2E_\bull = \Rop_2(\Rop_2E_\bull)$ is 
\begin{equation*}
\begin{tikzcd}
   &   & 4  & 3 & 2 & 1 \\
\Rop_{2}^2E_\bull:  & \cdots  \ar[r,two heads]  & E_{2}\ar[r,two heads] &  E_1  \ar[r,equals] & E_1 \ar[r,equals] & E_{1} \ar[r,two heads] & \cdots
\end{tikzcd}
\end{equation*}
which equals $\Rop_3\Rop_2E_\bull$.  (Writing out the definition, $(\Rop_k \Rop_\ell E_\bull)_i = (\Rop_\ell E_\bull)_{\Rop_k(i)} = E_{\Rop_\ell \Rop_k(i)}$.)

Assume $E_\bull$ is free and standard, so $\Sym^*(E_\bull) = \kk[Z_\bull] = \kk[z_{ij}]_{i\geq j}$.  Then the homomorphism $\Sym^*(E_\bull) \tto \Sym^*(\Rop_kE_\bull)$ is given by setting the diagonal entries $z_{ii}=0$ for $i\geq k$.  In particular, the homomorphism induced by $\Rop_{k+1}E_\bull \tto \Rop_kE_\bull$ simply sets the entry $z_{kk}$ to $0$.  It will be convenient to write $Z_k$ for the lower triangular matrix of variables in $\Sym^*(\Rop_kE_\bull)$.  For example, with $k=3$, the map $\Sym^*(\Rop_{k+1}E_\bull) \tto \Sym^*(\Rop_kE_\bull)$ looks as follows:
\begin{align*}
  \kk[Z_{4}] & \to \kk[Z_3] \\
{\left(\begin{array}{cccccc} z_{11}  \\ z_{21} & z_{22}  \\ z_{31} & z_{32} & \bm{z_{33}}  \\ z_{41} & z_{42} & z_{43} & 0  \\ z_{51} & z_{52} & z_{53} & z_{54} & 0  \\ \vdots &  &  &  &  & \ddots \end{array}\right)} & \to
{\left(\begin{array}{cccccc}
z_{11} &  &  &  &  &    \\
z_{21} & z_{22} &  &  &  &    \\
z_{31} & z_{32} & {\bm 0} &  &  &    \\
z_{41} & z_{42} & z_{43} & 0 &  &    \\
z_{51} & z_{52} & z_{53} & z_{54} & 0 &    \\
\vdots &  &  &  &  & \ddots 
\end{array}\right)}.
\end{align*}


\subsection{Action on single-column modules}

Assume $E_\bull$ is free and standard.  For a diagram with a single column $\ba = \{a_1<\cdots<a_r\}$, we have a basis for $E_\bull^{\ba}$ consisting of $z_{\ba,\bb}$ with $\bb\leq \ba$.  Similarly, $\Rop_kE_\bull^\ba$ has a basis consisting of those $z_{\ba,\bb}$ such that $\bb\leq \ba$, and furthermore $b_i<a_i$ whenever $a_i\geq k$.  This condition comes up frequently, so we introduce a shorthand for it:

\begin{definition}
Write $\bb\leq_k \ba$ to mean $b_i\leq a_i$ for all $i$, and $b_i<a_i$ whenever $a_i\geq k$.
\end{definition}

Under the projection $E_\bull^\ba \tto \Rop_kE_\bull^\ba$, the standard basis maps by
\[
  z_{\ba,\bb} \mapsto \begin{cases} z_{\ba,\bb} & \text{if } \bb \leq_k \ba; \\ 0 &\text{otherwise.} \end{cases}
\]
The same formula defines the map $\Rop_{k+1}E_\bull^\ba \tto \Rop_kE_\bull^\ba$.  It follows that the kernel of this map has a basis given by the nonzero $z_{\ba,\bb}$ which map to $0$ under the projection.  Stated formally:
\begin{lemma}\label{l.Na}
Assume $E_\bull$ is free and standard, and let $N_k^\ba = \ker( \Rop_{k+1}E_\bull^{\ba} \tto \Rop_kE_\bull^{\ba} )$.  Then $N_k^\ba$ is spanned by  $z_{\ba,\bb}$ such that $\bb\leq_{k+1}\ba$ and $a_i=b_i=k$ for some $i$. \qed
\end{lemma}

This is easily illustrated by considering $\Rop_{k+1}E_\bull^\ba$ as a submodule of $\kk[Z_{k+1}]$ via the embedding $z_{\ba,\bb} \mapsto \Delta_{\ba,\bb}$.  The point is that $\Delta_{\ba,\bb}$ maps to zero whenever $z_{kk}$ appears as a diagonal entry of the minor.

\begin{example}\label{ex.235r}
Consider $\ba=\{2,3,5\}$, as in Example~\ref{ex.235}.  A basis for $\Rop_4E_\bull^\ba$ is indexed by the column-strict flagged fillings $\bb$ whose entries in row $5$ are at most $4$:
\begin{center}
\pspicture(0,-5)(30,55)

\rput(-5,45){$1$}
\rput(-5,35){$2$}
\rput(-5,25){$3$}
\rput(-5,15){$4$}
\rput(-5,5){$5$}

\emptygraybox(10,40)
\emptygraybox(10,10)

\blankbox(10,30) \rput(15,35){$1$}
\blankbox(10,20) \rput(15,25){$2$}
\blankbox(10,0) \rput(15,5){$3$}

\endpspicture
\pspicture(0,-5)(30,55)

\emptygraybox(10,40)
\emptygraybox(10,10)

\blankbox(10,30) \rput(15,35){$1$}
\blankbox(10,20) \rput(15,25){$2$}
\blankbox(10,0) \rput(15,5){$4$}

\endpspicture
\pspicture(0,-5)(30,55)

\emptygraybox(10,40)
\emptygraybox(10,10)

\blankbox(10,30) \rput(15,35){$1$}
\blankbox(10,20) \rput(15,25){$3$}
\blankbox(10,0) \rput(15,5){$4$}

\endpspicture
\pspicture(0,-5)(30,55)

\emptygraybox(10,40)
\emptygraybox(10,10)

\blankbox(10,30) \rput(15,35){$2$}
\blankbox(10,20) \rput(15,25){$3$}
\blankbox(10,0) \rput(15,5){$4$}

\endpspicture
\end{center}
Only the first two survive the projection to $\Rop_3E_\bull^\ba$.  Identifying them with minors of $Z_4$, these four polynomials map as
\begin{align*}
\left|\begin{array}{ccc} z_{21} & z_{22} & 0 \\ z_{31} & z_{32} & \bm{z_{33}} \\ z_{51} & z_{52} & z_{53} \end{array}\right| 
&\mapsto \left|\begin{array}{ccc} z_{21} & z_{22} & 0 \\ z_{31} & z_{32} & \bm{0} \\ z_{51} & z_{52} & z_{53} \end{array}\right|, \\
\left|\begin{array}{ccc} z_{21} & z_{22} & 0 \\ z_{31} & z_{32} & 0 \\ z_{51} & z_{52} & z_{54} \end{array}\right| 
&\mapsto \left|\begin{array}{ccc} z_{21} & z_{22} & 0 \\ z_{31} & z_{32} & 0 \\ z_{51} & z_{52} & z_{54} \end{array}\right|, \\
\left|\begin{array}{ccc} z_{21} & 0 & 0 \\ z_{31} & \bm{z_{33}} & 0 \\ z_{51} & z_{53} & z_{54} \end{array}\right| 
&\mapsto \left|\begin{array}{ccc} z_{21} & 0 & 0 \\ z_{31} & \bm{0} & 0 \\ z_{51} & z_{53} & z_{54} \end{array}\right| =0, \\
\left|\begin{array}{ccc} z_{22} & 0 & 0 \\ z_{32} & \bm{z_{33}} & 0 \\ z_{52} & z_{53} & z_{54} \end{array}\right| 
&\mapsto \left|\begin{array}{ccc} z_{22} & 0 & 0 \\ z_{32} & \bm{0} & 0 \\ z_{52} & z_{53} & z_{54} \end{array}\right| =0.
\end{align*}
\end{example}

\subsection{Action on tensor product modules}

Still assuming $E_\bull$ is free and standard, the above considerations apply equally well to the tensor modules $TE_\bull^D = E_\bull^{\ba^{(1)}} \otimes \cdots \otimes E_\bull^{\ba^{(m)}}$.  In particular, the module $T(\Rop_kE_\bull)^D$ has a basis consisting of
\[
 z_\tau = z_{\ba^{(1)},\bb^{(1)}} \otimes \cdots \otimes z_{\ba^{(m)},\bb^{(m)}},
\]
indexed by column-strict flagged fillings $\tau$ such that $\bb^{(j)} \leq_k \ba^{(j)}$ for all $j$.  The natural extension of Lemma~\ref{l.Na} follows from the single-column case:
\begin{lemma}\label{l.TN}
Assume $E_\bull$ is free and standard, and let $TN_k^D = \ker\left( T(\Rop_{k+1}E_\bull)^D \tto T(\Rop_{k}E_\bull)^D \right)$.  Then $TN_k^D$ is spanned by $z_\tau = z_{\ba^{(1)},\bb^{(1)}} \otimes \cdots \otimes z_{\ba^{(m)},\bb^{(m)}}$ such that $\bb^{(j)} \leq_{k+1} \ba^{(j)}$ for all $j$, and $a^{(j)}_i = b^{(j)}_i = k$ for some $i$ and some $j$. \qed
\end{lemma}

As before, the elements generating this kernel are precisely those $z_\tau$ which map to $0$ under the projection.  They can be interpreted as products of minors $\Delta_\tau = \prod_{j=1}^m \Delta_{\ba^{(j)},\bb^{(j)}}$ such that some factor $\Delta_{\ba^{(j)},\bb^{(j)}}$ becomes zero when $z_{kk}=0$.

\medskip

We will compute the kernels of the maps $\Rop_{k+1}E_\bull^D \tto \Rop_kE_\bull^D$ in \S\ref{s.kernel}.

\subsection{Action on characters}

Assume $E_\bull$ is free and standard, and furthermore that $\kk$ is a field.  Simplifying notation, let $z_1,\ldots,z_n$ be a basis for $E_n$, and let $E_i$ be the span of $z_1,\ldots,z_i$, regarded as a quotient $E_n\tto E_i$ by the projection which sends $z_j \mapsto 0$ for $j>i$.  The next lemma shows that definition of $\Rop_k$ as an operator on flags and modules is compatible with its action on polynomials.

\begin{lemma}\label{l.rk-character}
For any diagram $D$, we have
\[
\Rop_k\tchar_T(E_\bull^D) = \tchar_T(\Rop_kE_\bull^D).
\]
\end{lemma}

\begin{proof}
For convenience, given a multi-index $\lambda = (\lambda_1,\lambda_2,\ldots)$, we use the temporary notation $\hat{\Rop}_k\lambda = (\lambda_1,\ldots,\lambda_{k-1},0,\lambda_k,\ldots)$.  Using this notation, the action of the Bergeron-Sottile operator on a polynomial $f \in \Z[x_1,x_2,\ldots]$, is characterized by
\[
  [x^\lambda](\Rop_kf) = [x^{\hat{\Rop}_k\lambda}](f),
\]
where $[x^\mu](f)$ is the coefficient of $x^\mu$ in $f$.

The lemma is equivalent to the assertion
\[
  \dim( \Rop_kE_\bull^D )_{\lambda} = \dim( E_\bull^D )_{\hat{\Rop}_k\lambda},
\]
where $(\cdot)_\mu$ is the weight space where $T$ acts by character $\mu$.  This is what we will prove.

Consider the surjection of flags $\pi\colon E_\bull \tto \Rop_kE_\bull$ determined by
\[
  z_i \mapsto \begin{cases} z_i & \text{if }i<k; \\ 0 & \text{if }i=k; \\ z_{i-1} & \text{if }i>k.\end{cases}
\]
This is equivariant with respect to the homomorphism $T\to T$ given by
\[
 (t_1,t_2,\ldots) \mapsto (t_1,\ldots,t_{k-1},t_{k+1},\ldots), 
\]
and it follows that the induced surjection $\pi^D\colon E_\bull^D \tto \Rop_kE_\bull^D$ acts on weight spaces by sending $(E_\bull^D)_\mu$ to $(\Rop_kE_\bull^D)_\lambda$ whenever $\mu_i=\lambda_i$ for $i<k$ and $\mu_i=\lambda_{i+1}$ for $i\geq k$.  In fact, given any flagged filling $\tau$ of $D$, with corresponding weight vector $\Delta_\tau$, we have $\pi^D(\Delta_\tau)=0$ if and only if $k$ appears in $\tau$; it follows that $\pi^D$ sends the weight space $(E_\bull^D)_\mu$ to $0$ if and only if $\mu_k>0$. Furthermore, $\pi^D$ restricts to a surjection $(E_\bull^D)_{\hat{\Rop}\lambda} \tto (\Rop_kE_\bull^D)_\lambda$ (as can be seen by constructing a section).  Combined, these observations imply that $(E_\bull^D)_{\hat{\Rop}\lambda} \to (\Rop_kE_\bull^D)_\lambda$ is an isomorphism.
\end{proof}

\begin{remark}
The quotient $E_\bull^D \tto \Rop_kE_\bull^D$ induced by the standard projection $E_\bull \to \Rop_kE_\bull$ arises from the quotient of the polynomial ring $\C[Z_\bull]$ which sets the diagonal entries $z_{ii}=0$ for $i\geq k$.  Analogously, the quotient $\pi^D$ used in the proof of Lemma~\ref{l.rk-character} comes from the quotient of $\C[Z_\bull]$ which sets the $k$th column to zero, and shifts the other columns to the right of the $k$th column one step to the left: $z_{ik}\mapsto0$ for $i\geq k$, and $z_{ij}\mapsto z_{i,j-1}$ for $j>k$.
\end{remark}

\section{The kernel of $\Rop_{k+1} \tto \Rop_k$}\label{s.kernel}

For the rest of the article, we assume $E_\bull$ is free and standard.


The general claim we need is a description of the kernel of the $\Rop_{k+1} \tto \Rop_k$ projection for diagram modules $E_\bull^D$, analogous to the ones we saw above for $E_\bull^\ba$ and $TE_\bull^D$.  Before stating it, we record a basic fact about pullbacks and pushouts.

\begin{proposition}\label{p.app-pushout}
Let $\kk$ be a commutative ring, and consider the following commutative diagram of $\kk$-modules, in which the top row is right-exact and the bottom row is left-exact:
\begin{equation}\label{e.pushout}
\begin{tikzcd}
 & A \ar[r,"f"] \ar[d,"\alpha"] & B \ar[r,"g"] \ar[d, "\beta"] & C \ar[r] \ar[d,"\gamma"] & 0 \\
0 \ar[r] & A' \ar[r,"f'"]  & B' \ar[r,"g'"]  & C'   .
\end{tikzcd}
\end{equation}
\begin{enumerate}
\item If the left-hand square (involving $A,A',B,B'$) is a pullback square, then $\gamma$ is injective.

\item If the right-hand square (involving $B,B',C,C'$) is a pushout square, then $\alpha$ is surjective.
\end{enumerate}
\end{proposition}

The proof is a straightforward exercise.  The second statement of the proposition establishes the following:

\begin{lemma}\label{l.ker-gen}
Let $N_k^D = \ker( \Rop_{k+1}E_\bull^D \tto \Rop_kE_\bull^D )$, considered as a submodule of $\kk[Z_{k+1}]$.  Then $N_k^D$ is generated by the products $\Delta_\tau=\Delta_{\ba^{(1)},\bb^{(1)}}\cdots\Delta_{\ba^{(m)},\bb^{(m)}}$ which map to zero.  That is, for some $j$ and some $i$, we have $a^{(j)}_i=b^{(j)}_i=k$.
\end{lemma}

\begin{proof}
From the definitions, we have a diagram
\begin{equation}
\begin{tikzcd}
 & TN^D_k \ar[r] \ar[d, "\alpha"] & T(\Rop_{k+1}E_\bull)^D \ar[r] \ar[d, two heads] & T(\Rop_kE_\bull)^D \ar[r] \ar[d, two heads] & 0 \\
0 \ar[r] & N^D_k \ar[r] & \Rop_{k+1}E_\bull^D \ar[r] & \Rop_kE_\bull^D,
\end{tikzcd}
\end{equation}
in which the rows are exact, and the right-hand square is a push-out.  It follows from the above proposition that $\alpha$ is surjective.  Since $TN_k^D$ is generated by the nonzero $z_\tau$ which map to zero in $T(\Rop_{k+1}E_\bull)^D$ (Lemma~\ref{l.TN}), and $\alpha(z_\tau) = \Delta_\tau$, the assertion is proved.
\end{proof}

For the main theorems, we require two further facts, which follow from this lemma.

\begin{definition}
A diagram $D$ is said to be \define{$k$-full}\footnote{This is a variation on the notion of {\it $i$-free diagram}, used by Magyar \cite{magyar}.} if whenever a box appears in row $k$, there is a box below it in row $k+1$.  That is, for each column $\ba^{(j)}$ in $D$, either $k\not\in \ba^{(j)}$, or else $\{k,k+1\}\subset \ba^{(j)}$.
\end{definition}

\begin{lemma}\label{l.k-full}
If $D$ is $k$-full then $\Rop_{k+1}E_\bull^D \tto \Rop_kE_\bull^D$ is an isomorphism.  (That is, $N^D_k=0$.)
\end{lemma}

\begin{proof}
When $D$ is $k$-full, having $a_i^{(j)}=k$ implies $a_{i+1}^{(j)}=k+1$, and having $b_i^{(j)}=k$ implies $b_{i+1}^{(j)}\geq k+1$.  This violates the flagging condition for $\Rop_{k+1}E^D_\bull$, so the corresponding minor $\Delta_{\ba^{(j)},\bb^{(j)}}(Z_{k+1})$ was already zero, and Lemma~\ref{l.ker-gen} implies $N^D_k=0$.
\end{proof}

\begin{example}\label{ex.lemma1}
Let $D$ be the following diagram:
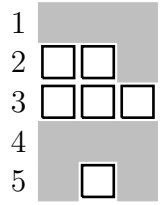
\begin{center}
\pspicture(0,0)(30,60)

\rput(-5,45){$1$}
\rput(-5,35){$2$}
\rput(-5,25){$3$}
\rput(-5,15){$4$}
\rput(-5,5){$5$}

\emptygraybox(0,40)
\emptygraybox(0,10)
\emptygraybox(0,0)

\blankbox(0,30)
\blankbox(0,20)


\emptygraybox(10,40)
\emptygraybox(10,10)

\blankbox(10,30)
\blankbox(10,20)
\blankbox(10,0)

\emptygraybox(20,40)
\emptygraybox(20,30)
\emptygraybox(20,10)
\emptygraybox(20,0)

\blankbox(20,20)

\endpspicture
\end{center}
%
This diagram is $k$-full for $k=1$, $k=2$, and $k=4$, but not for $k=3$ or $k=5$.  Consider $k=2$, so the lemma asserts that $\Rop_3E_\bull^D \to \Rop_2E_\bull^D$ is an isomorphism.  For $\Rop_3E_\bull^D$, we have a basis of $\Delta_\tau$ as $\tau$ varies over the following four fillings:
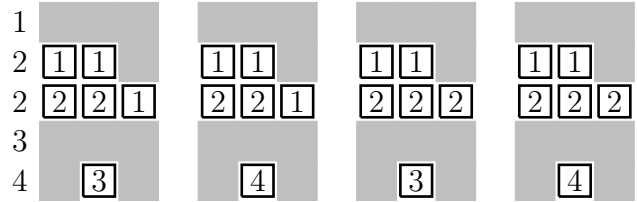
\begin{center}
\pspicture(0,0)(30,60)

\rput(-5,45){$1$}
\rput(-5,35){$2$}
\rput(-5,25){$2$}
\rput(-5,15){$3$}
\rput(-5,5){$4$}

\emptygraybox(0,40)
\emptygraybox(0,10)
\emptygraybox(0,0)

\blankbox(0,30) \rput(5,35){$1$}
\blankbox(0,20) \rput(5,25){$2$}


\emptygraybox(10,40)
\emptygraybox(10,10)

\blankbox(10,30) \rput(15,35){$1$}
\blankbox(10,20) \rput(15,25){$2$}
\blankbox(10,0)  \rput(15,5){$3$}

\emptygraybox(20,40)
\emptygraybox(20,30)
\emptygraybox(20,10)
\emptygraybox(20,0)

\blankbox(20,20)  \rput(25,25){$1$}

\endpspicture
\pspicture(-10,0)(30,60)

\emptygraybox(0,40)
\emptygraybox(0,10)
\emptygraybox(0,0)

\blankbox(0,30) \rput(5,35){$1$}
\blankbox(0,20) \rput(5,25){$2$}


\emptygraybox(10,40)
\emptygraybox(10,10)

\blankbox(10,30) \rput(15,35){$1$}
\blankbox(10,20) \rput(15,25){$2$}
\blankbox(10,0)  \rput(15,5){$4$}

\emptygraybox(20,40)
\emptygraybox(20,30)
\emptygraybox(20,10)
\emptygraybox(20,0)

\blankbox(20,20)  \rput(25,25){$1$}

\endpspicture
\pspicture(-10,0)(30,60)

\emptygraybox(0,40)
\emptygraybox(0,10)
\emptygraybox(0,0)

\blankbox(0,30) \rput(5,35){$1$}
\blankbox(0,20) \rput(5,25){$2$}


\emptygraybox(10,40)
\emptygraybox(10,10)

\blankbox(10,30) \rput(15,35){$1$}
\blankbox(10,20) \rput(15,25){$2$}
\blankbox(10,0)  \rput(15,5){$3$}

\emptygraybox(20,40)
\emptygraybox(20,30)
\emptygraybox(20,10)
\emptygraybox(20,0)

\blankbox(20,20)  \rput(25,25){$2$}

\endpspicture
\pspicture(-10,0)(30,60)

\emptygraybox(0,40)
\emptygraybox(0,10)
\emptygraybox(0,0)

\blankbox(0,30) \rput(5,35){$1$}
\blankbox(0,20) \rput(5,25){$2$}


\emptygraybox(10,40)
\emptygraybox(10,10)

\blankbox(10,30) \rput(15,35){$1$}
\blankbox(10,20) \rput(15,25){$2$}
\blankbox(10,0)  \rput(15,5){$4$}

\emptygraybox(20,40)
\emptygraybox(20,30)
\emptygraybox(20,10)
\emptygraybox(20,0)

\blankbox(20,20)  \rput(25,25){$2$}

\endpspicture
\end{center}
The corresponding $\Delta_\tau$ are:
\begin{align*}
 \left|\begin{array}{cc} z_{21} & \bm{z_{22}} \\ z_{31} & z_{32} \end{array}\right|
\cdot
\left|\begin{array}{ccc} z_{21} & \bm{z_{22}} & 0 \\ z_{31} & z_{32} & 0 \\ z_{51} & z_{52} & z_{53} \end{array}\right|
\cdot
z_{31} \\
 \left|\begin{array}{cc} z_{21} & \bm{z_{22}} \\ z_{31} & z_{32} \end{array}\right|
\cdot
\left|\begin{array}{ccc} z_{21} & \bm{z_{22}} & 0 \\ z_{31} & z_{32} & 0 \\ z_{51} & z_{52} & z_{54} \end{array}\right|
\cdot
z_{31} \\
 \left|\begin{array}{cc} z_{21} & \bm{z_{22}} \\ z_{31} & z_{32} \end{array}\right|
\cdot
\left|\begin{array}{ccc} z_{21} & \bm{z_{22}} & 0 \\ z_{31} & z_{32} & 0 \\ z_{51} & z_{52} & z_{53} \end{array}\right|
\cdot
z_{32} \\
 \left|\begin{array}{cc} z_{21} & \bm{z_{22}} \\ z_{31} & z_{32} \end{array}\right|
\cdot
\left|\begin{array}{ccc} z_{21} & \bm{z_{22}} & 0 \\ z_{31} & z_{32} & 0 \\ z_{51} & z_{52} & z_{54} \end{array}\right|
\cdot
z_{32}.
\end{align*}
The map $\Rop_3 \to \Rop_2$ is given by $z_{22}\mapsto 0$, so these four polynomials map to
\begin{align*}
 \left|\begin{array}{cc} z_{21} & \bm{0} \\ z_{31} & z_{32} \end{array}\right|
\cdot
\left|\begin{array}{ccc} z_{21} & \bm{0} & 0 \\ z_{31} & z_{32} & 0 \\ z_{51} & z_{52} & z_{53} \end{array}\right|
\cdot
z_{31} \\
 \left|\begin{array}{cc} z_{21} & \bm{0} \\ z_{31} & z_{32} \end{array}\right|
\cdot
\left|\begin{array}{ccc} z_{21} & \bm{0} & 0 \\ z_{31} & z_{32} & 0 \\ z_{51} & z_{52} & z_{54} \end{array}\right|
\cdot
z_{31} \\
 \left|\begin{array}{cc} z_{21} & \bm{0} \\ z_{31} & z_{32} \end{array}\right|
\cdot
\left|\begin{array}{ccc} z_{21} & \bm{0} & 0 \\ z_{31} & z_{32} & 0 \\ z_{51} & z_{52} & z_{53} \end{array}\right|
\cdot
z_{32} \\
 \left|\begin{array}{cc} z_{21} & \bm{0} \\ z_{31} & z_{32} \end{array}\right|
\cdot
\left|\begin{array}{ccc} z_{21} & \bm{0} & 0 \\ z_{31} & z_{32} & 0 \\ z_{51} & z_{52} & z_{54} \end{array}\right|
\cdot
z_{32}.
\end{align*}
So this map is indeed an isomorphism---but not the identity!
\end{example}


\begin{remark}
Lemma~\ref{l.k-full} can be proved directly using standard monomial theory, which makes a pleasant exercise.  
\end{remark}

\begin{definition}
A diagram $D$ has a {\it descent} at $k$ if, possibly after re-ordering columns, there is a column $\ba^{(j)}$ such that
\begin{itemize}
\item $\ba^{(j)} = \{a_1^{(j)} < \cdots < a_r^{(j)}=k\}$, i.e., the last box in this column is in row $k$;
\item for $j'<j$, the columns $\ba^{(j')}$ (i.e., to the left of $\ba^{(j)}$) are all $k$-full; and
\item for $j'>j$, the columns $\ba^{(j')}$ (i.e., to the right of $\ba^{(j)}$) all have $\ba^{(j')}_{\leq k+1} \subseteq \ba^{(j)}$, where $\ba_{\leq k+1}$ is the subset $\ba \cap \{1,\ldots,k+1\}$, so in particular $k+1 \not\in \ba^{(j')}$.
\end{itemize}
\end{definition}

If $D$ has a descent at $k$, the witnessing column $\ba^{(j)}$ is unique up to repetition.  (That is, if both $\ba^{(j)}$ and $\ba^{(j')}$ are witnesses, then $\ba^{(j)}=\ba^{(j')}$ as sets.)  By convention, we will choose the leftmost witnessing column for a descent.  Following \cite{fgrs}, the box in position $(k,j)$ is called a \define{border cell}.

Given a diagram $D$ with a descent at $k$, let $s_kD$ be the diagram obtained by deleting the border cell $(k,j)$ and then swapping rows $k$ and $k+1$.

\begin{example}\label{ex.descent}
The diagram from Example~\ref{ex.lemma1} has a descents at $k=3$ and $k=5$, with the border cells marked:
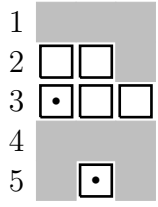
\begin{center}
\pspicture(0,0)(30,60)

\rput(-5,45){$1$}
\rput(-5,35){$2$}
\rput(-5,25){$3$}
\rput(-5,15){$4$}
\rput(-5,5){$5$}

\emptygraybox(0,40)
\emptygraybox(0,10)
\emptygraybox(0,0)

\blankbox(0,30)
\blankbox(0,20) \rput(5,25){$\bull$}


\emptygraybox(10,40)
\emptygraybox(10,10)

\blankbox(10,30)
\blankbox(10,20)
\blankbox(10,0) \rput(15,5){$\bull$}

\emptygraybox(20,40)
\emptygraybox(20,30)
\emptygraybox(20,10)
\emptygraybox(20,0)

\blankbox(20,20)

\endpspicture
\end{center}
%
The corresponding diagrams $s_3D$ and $s_5D$ are shown below.

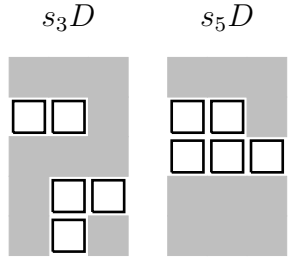
\begin{center}

\pspicture(0,0)(40,70)

\rput(15,60){$s_3D$}

\emptygraybox(0,40)
\emptygraybox(0,20)
\emptygraybox(0,10)
\emptygraybox(0,0)

\blankbox(0,30)


\emptygraybox(10,40)
\emptygraybox(10,20)

\blankbox(10,30)
\blankbox(10,10)
\blankbox(10,0)

\emptygraybox(20,40)
\emptygraybox(20,30)
\emptygraybox(20,20)
\emptygraybox(20,0)

\blankbox(20,10)

\endpspicture
\pspicture(0,0)(30,70)

\rput(15,60){$s_5D$}

\emptygraybox(0,40)
\emptygraybox(0,10)
\emptygraybox(0,0)

\blankbox(0,30)
\blankbox(0,20) 


\emptygraybox(10,40)
\emptygraybox(10,10)
\emptygraybox(10,0)

\blankbox(10,30)
\blankbox(10,20)


\emptygraybox(20,40)
\emptygraybox(20,30)
\emptygraybox(20,10)
\emptygraybox(20,0)

\blankbox(20,20)

\endpspicture

\end{center}
\end{example}

\begin{example}
It can happen that $k$ is a descent of $s_kD$.  For instance, with $D=\big[\{1\},\{2\}\big]$, we have $s_1D=\big[\{1\}\big]$ and $s_1 s_1 D = \emptyset$.
\end{example}

\begin{lemma}\label{l.k-descent}
If $D$ has a descent at $k$, then
\[
  N_k^D \isom K_k \otimes \Rop_{k+1}E_\bull^{s_kD},
\]
where $K_k= \ker(\Rop_{k+1}E_{k}\tto \Rop_kE_{k}) = \ker(E_k \tto E_{k-1})$.
\end{lemma}

In our setting, $K_k$ is the one-dimensional $B$-module with character $x_k$.  
To prove the lemma, we will use the classical quadratic straightening relations.

\begin{proof}
Suppose $\ba^{(j)} = \{a_1^{(j)} < \cdots < a_r^{(j)}=k \}$ is the column witnessing the descent, so  for all $j'<j$, $\ba^{(j')}$ is $k$-full, and for all $j'>j$, we have $\ba^{(j')}_{\leq k+1} \subseteq \ba^{(j)}$.  Further re-ordering if necessary, we may assume that {\it all} of the $k$-full columns are to the left of $\ba^{(j)}$; that is, for every $j'\geq j$, we have $\ba^{(j')} \cap \{k,k+1\} = \{k\}$.

We know that $N_k^D$ is generated by the nonzero $\Delta_\tau(Z_{k+1})$ such that the filling $\tau$ puts an entry ``$k$'' in the $k$th row; that is, writing $\Delta_\tau = \Delta_{\ba^{(1)},\bb^{(1)}} \cdots \Delta_{\ba^{(m)},\bb^{(m)}}$, we have $a^{(j')}_i=b^{(j')}_i=k$ for some $i,j'$.  In our setting, such an entry can occur in a column $\ba^{(j')}$ with $j'\geq j$, and only in these columns.  (Placing a ``$k$'' in the $k$th row of a $k$-full column makes $\Delta_\tau(Z_{k+1})=0$.)  

\medskip
\noindent
{\bf Claim.}  The $\Delta_\tau$ such that $b^{(j)}_r=k$ suffice to generate $N_k^D$.

\medskip

Specifically, suppose row $k$ in column $j'$ is filled by $k$, for some $j'>j$; that is, $a^{(j')}_i=b^{(j')}_i=k$.  We will see that
\begin{equation}
  \Delta_{\ba^{(j)},\bb^{(j)}}\cdot \Delta_{\ba^{(j')},\bb^{(j')}} = \sum_{\bb,\bb'} \pm \Delta_{\ba^{(j)},\bb}\cdot \Delta_{\ba^{(j')},\bb'},
\end{equation}
a sum over fillings $\bb,\bb'$ such that $\bb$ places $k$ in the $k$th row of the $j$th column; that is, $a^{(j)}_r=b_r=k$.  If $\ba^{(j)}$ already has $k$ in this position, there is nothing to prove, so for the rest of the proof we assume $\ba^{(j)}$ does not contain $k$.  Together with the flagging condition, this implies all entries in $\ba^{(j)}$ are strictly less than $k$.

Analogously to the situation explained in \S\ref{ss.d-exchange}, we can represent any minor $\Delta_{\ba,\bb}(Z_{k+1})$ as a maximal minor of the augmented matrix $\hat{Z}_{k+1} = (Z_{k+1}\,|\,I)$.  Selecting the entry ``$k$'' in the $k$th row of the column $\ba^{(j')}$, consider the Sylvester exchange relations
\[
  \Delta_{\ba^{(j)},\bb^{(j)}} \cdot \Delta_{\ba^{(j')},\bb^{(j')}} = \sum_{\bb,\bb'} \Delta_{\ba^{(j)},\bb}\cdot \Delta_{\ba^{(j')},\bb'},
\]
the sum over $\bb,\bb'$ obtained by exhanging the entry in position $(k,j')$ with an entry in position $(i,j)$ (for all $i$).  We describe the term $\Delta_{\ba^{(j)},\bb}\cdot \Delta_{\ba^{(j')},\bb'}$ by distinguishing three cases:
\begin{enumerate}
\item $i\not\in\ba^{(j)}$ and $i<k$. Since $\ba^{(j)}\supseteq \ba^{(j')}_{\leq k+1}$, this means $i\not\in\ba^{(j')}$, and that the corresponding entries of $\bb^{(j)}$ and $\bb^{(j')}$ are both $\textcolor{blue}{\hat\imath}$.  Upon performing the exchange, $\bb'$ has two copies of $\textcolor{blue}{\hat\imath}$, so the corresponding minor $\Delta_{\ba^{(j')},\bb'}$ vanishes.\label{pfcase1}

\item $i\in\ba^{(j)}$, so $i\leq k$.  By assumption, all entries of $\bb^{(j)}$ are less than $k$.  So after performing the exchange, $\bb$ has largest entry equal to $k$; after re-arranging entries to increasing order (and possibly introducing a sign in the process), this must appear in the bottom row of $\ba^{(j)}$, which is row $k$.  So in this case, we obtain a term $\pm\Delta_{\ba^{(j)},\bb}\cdot \Delta_{\ba^{(j')},\bb'}$ of the desired form.\label{pfcase2}

\item $i>k$.  Such an entry of $\bb^{(j)}$ below row $k$ is equal to $\textcolor{blue}{\hat\imath}$.  After exchanging it with the entry $k$ in row $k$ of $\bb^{(j')}$ to obtain $\bb$, either we get a repeated $\textcolor{blue}{\hat\imath}$ in $\bb'$ again (so the minor is zero), or else---after re-arranging to increasing order---the entries of $\bb^{(j')}$ in rows $k+1$ through $i$ get shifted up, so that $\bb'$ has an entry in row $k$, equal to the first entry below row $k$ of $\bb^{(j')}$.  Since $\bb^{(j')}$ had $k$ in row $k$, all the entries below row $k$ are strictly greater than $k$, meaning $\bb'$ has an entry $>k$ in row $k$.  Therefore $\Delta_{\ba^{(j')},\bb'}(Z_{k+1})=0$, by the flagging condition.\label{pfcase3}
\end{enumerate}
To summarize: the only nonzero terms are those in the second case, and these are of the asserted form, so the claim is proved.

The claim shows that $N_k^D$ is generated by $\Delta_\tau$ having some factor divisible by $z_{kk}$.  Pulling this factor out, we have
\[
  N_k^D = K_k \otimes \Rop_{k+1}E_\bull^{D'}
\]
where $D' = D \setminus (k,j)$ is the diagram obtained by removing the border cell witnessing the descent at $k$.  Under the operator $\Rop_{k+1}$, the flagging conditions $\Rop_{k+1}E_\bull^{D'}$ treat rows $k$ and $k+1$ identically, so they may be swapped.  Performing the swap, we obtain $\Rop_{k+1}E_\bull^{D'}=\Rop_{k+1}E_\bull^{s_kD}$, completing the proof of the lemma.
\end{proof}

\begin{example}
Consider $\ba^{(j)}=\{2,3\}$ and $\ba^{(j')}=\{2,3,5\}$, with $k=3$, as in \S\ref{ss.d-exchange}, \eqref{e.2column-relation}. When the minors in the relation \eqref{e.2column-relation} are restricted to $Z_{4}$, only one term survives on the right-hand side, and we have
\begin{equation}\label{e.2column-relation-zdot}
\begin{aligned}
\raisebox{-35pt}{
\pspicture(0,0)(20,50)

\emptygraybox(0,40)
\emptygraybox(0,10)
\emptygraybox(0,0)

\blankbox(0,30) \rput(5,35){$1$}
\blankbox(0,20) \rput(5,25){$2$}


\emptygraybox(10,40)
\emptygraybox(10,10)

\blankbox(10,30) \rput(15,35){$1$}
\blankbox(10,20) \rput(15,25){$\bm{3}$}
\blankbox(10,0)  \rput(15,5){$5$}

\endpspicture
}
&=
\raisebox{-35pt}{
\pspicture(0,0)(20,50)

\emptygraybox(0,10) \rput(5,15){\small$\textcolor{blue}{\hat{4}}$}
\emptygraybox(0,0) \rput(5,5){\small$\textcolor{blue}{\hat{5}}$}

\blankbox(0,40) \rput(5,45){$\bm{3}$}
\blankbox(0,30) \rput(5,35){$1$}
\blankbox(0,20) \rput(5,25){$2$}


\emptygraybox(10,40) \rput(15,45){\small$\textcolor{blue}{\hat{1}}$}
\emptygraybox(10,20) \rput(15,25){\small$\textcolor{blue}{\hat{1}}$}
\emptygraybox(10,10) \rput(15,15){\small$\textcolor{blue}{\hat{4}}$}

\blankbox(10,30) \rput(15,35){$1$}
\blankbox(10,0)  \rput(15,5){$5$}

\endpspicture
}
+
\raisebox{-35pt}{
\pspicture(0,0)(20,50)

\emptygraybox(0,40) \rput(5,45){\small$\textcolor{blue}{\hat{1}}$}
\emptygraybox(0,10) \rput(5,15){\small$\textcolor{blue}{\hat{4}}$}
\emptygraybox(0,0) \rput(5,5){\small$\textcolor{blue}{\hat{5}}$}

\blankbox(0,30) \rput(5,35){$\bm{3}$}
\blankbox(0,20) \rput(5,25){$2$}


\emptygraybox(10,40) \rput(15,45){\small$\textcolor{blue}{\hat{1}}$}
\emptygraybox(10,10) \rput(15,15){\small$\textcolor{blue}{\hat{4}}$}

\blankbox(10,30) \rput(15,35){$1$}
\blankbox(10,20) \rput(15,25){$1$}
\blankbox(10,0)  \rput(15,5){$5$}

\endpspicture
}
+
\raisebox{-35pt}{
\pspicture(0,0)(20,50)

\emptygraybox(0,40) \rput(5,45){\small$\textcolor{blue}{\hat{1}}$}
\emptygraybox(0,10) \rput(5,15){\small$\textcolor{blue}{\hat{4}}$}
\emptygraybox(0,0) \rput(5,5){\small$\textcolor{blue}{\hat{5}}$}

\blankbox(0,30) \rput(5,35){$1$}
\blankbox(0,20) \rput(5,25){$\bm{3}$}


\emptygraybox(10,40) \rput(15,45){\small$\textcolor{blue}{\hat{1}}$}
\emptygraybox(10,10) \rput(15,15){\small$\textcolor{blue}{\hat{4}}$}

\blankbox(10,30) \rput(15,35){$1$}
\blankbox(10,20) \rput(15,25){$2$}
\blankbox(10,0)  \rput(15,5){$5$}

\endpspicture
}
+
\raisebox{-35pt}{
\pspicture(0,0)(20,50)

\emptygraybox(0,40) \rput(5,45){\small$\textcolor{blue}{\hat{1}}$}
\emptygraybox(0,0) \rput(5,5){\small$\textcolor{blue}{\hat{5}}$}

\blankbox(0,30) \rput(5,35){$1$}
\blankbox(0,20) \rput(5,25){$2$}
\blankbox(0,10) \rput(5,15){${\bm{3}}$}


\emptygraybox(10,40) \rput(15,45){\small$\textcolor{blue}{\hat{1}}$}
\emptygraybox(10,20) \rput(15,25){\small$\textcolor{blue}{\hat{4}}$}
\emptygraybox(10,10) \rput(15,15){\small$\textcolor{blue}{\hat{4}}$}

\blankbox(10,30) \rput(15,35){$1$}
\blankbox(10,0)  \rput(15,5){$5$}

\endpspicture
}
+
\raisebox{-35pt}{
\pspicture(0,0)(20,50)

\emptygraybox(0,40) \rput(5,45){\small$\textcolor{blue}{\hat{1}}$}
\emptygraybox(0,10) \rput(5,15){\small$\textcolor{blue}{\hat{4}}$}

\blankbox(0,30) \rput(5,35){$1$}
\blankbox(0,20) \rput(5,25){$2$}
\blankbox(0,0) \rput(5,5){${\bm{3}}$}

\emptygraybox(10,40) \rput(15,45){\small$\textcolor{blue}{\hat{1}}$}
\emptygraybox(10,20) \rput(15,25){\small$\textcolor{blue}{\hat{5}}$}
\emptygraybox(10,10) \rput(15,15){\small$\textcolor{blue}{\hat{4}}$}

\blankbox(10,30) \rput(15,35){$1$}
\blankbox(10,0)  \rput(15,5){$5$}

\endpspicture
}
\\
&=
%
\quad\; 0\quad + \quad\;0\quad +
\raisebox{-35pt}{
\pspicture(0,0)(20,50)

\emptygraybox(0,40) 
\emptygraybox(0,10) 
\emptygraybox(0,0) 

\blankbox(0,30) \rput(5,35){$1$}
\blankbox(0,20) \rput(5,25){$\bm{3}$}


\emptygraybox(10,40) 
\emptygraybox(10,10) 

\blankbox(10,30) \rput(15,35){$1$}
\blankbox(10,20) \rput(15,25){$2$}
\blankbox(10,0)  \rput(15,5){$5$}

\endpspicture
}
+ \quad0\quad + \quad 0.
\end{aligned}
\end{equation}
The vanishing of the last term is an instance of Case~\eqref{pfcase3} examined in the above proof.
\end{example}

\begin{example}\label{ex.lemma2}
Continuing to examine the diagram from Example~\ref{ex.descent}, consider the descent at $k=3$.
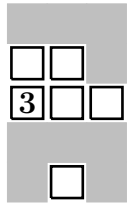
\begin{center}
\pspicture(0,0)(30,60)

\emptygraybox(0,40)
\emptygraybox(0,10)
\emptygraybox(0,0)

\blankbox(0,30)
\blankbox(0,20) \rput(5,25){$\bm{3}$}


\emptygraybox(10,40)
\emptygraybox(10,10)

\blankbox(10,30)
\blankbox(10,20)
\blankbox(10,0)

\emptygraybox(20,40)
\emptygraybox(20,30)
\emptygraybox(20,10)
\emptygraybox(20,0)

\blankbox(20,20)

\endpspicture
\end{center}
%
Flagged column-strict fillings of $D$ with $\bm{3}$ in the indicated border cell are evidently in bijection with flagged column-strict fillings of $D'$, the diagram obtained by removing the border cell from $D$.  
Furthermore, fillings of $D'$ which respect the flagging condition for $\Rop_{k+1}$ are in evident bijection with those of $s_3D$.  The $B$-modules $\Rop_4E_\bull^{D'}$ and $\Rop_4E_\bull^{s_3D}$ are isomorphic under the corresponding identification of minors.  For instance, the fillings
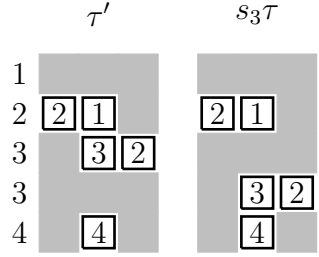
\begin{center}
\pspicture(0,0)(40,70)

\rput(-5,45){$1$}
\rput(-5,35){$2$}
\rput(-5,25){$3$}
\rput(-5,15){$3$}
\rput(-5,5){$4$}

\rput(15,60){$\tau'$}

\emptygraybox(0,40)
\emptygraybox(0,20)
\emptygraybox(0,10)
\emptygraybox(0,0)

\blankbox(0,30) \rput(5,35){$2$}


\emptygraybox(10,40)
\emptygraybox(10,10)

\blankbox(10,30)  \rput(15,35){$1$}
\blankbox(10,20)  \rput(15,25){$3$}
\blankbox(10,0)  \rput(15,5){$4$}

\emptygraybox(20,40)
\emptygraybox(20,30)
\emptygraybox(20,10)
\emptygraybox(20,0)

\blankbox(20,20)  \rput(25,25){$2$}

\endpspicture
\pspicture(0,0)(30,70)

\rput(15,60){$s_3\tau$}

\emptygraybox(0,40)
\emptygraybox(0,20)
\emptygraybox(0,10)
\emptygraybox(0,0)

\blankbox(0,30)  \rput(5,35){$2$}


\emptygraybox(10,40)
\emptygraybox(10,20)

\blankbox(10,30)  \rput(15,35){$1$}
\blankbox(10,10)  \rput(15,15){$3$}
\blankbox(10,0)  \rput(15,5){$4$}

\emptygraybox(20,40)
\emptygraybox(20,30)
\emptygraybox(20,20)
\emptygraybox(20,0)

\blankbox(20,10)  \rput(25,15){$2$}

\endpspicture
\end{center}
correspond to products of minors
\begin{align*}
\Delta_{\tau'}(Z_4) &= z_{22}
\cdot
\left|\begin{array}{ccc} z_{21} & 0 & 0 \\ z_{31} & z_{33} & 0 \\ z_{51} & z_{53} & z_{54} \end{array}\right|
\cdot
z_{32} \\
\intertext{and}
\Delta_{s_3\tau}(Z_4) &= z_{22}
\cdot
\left|\begin{array}{ccc} z_{21} & 0 & 0 \\ z_{41} & z_{43} & {\bm0} \\ z_{51} & z_{53} & z_{54} \end{array}\right|
\cdot
z_{42}
\end{align*}
(In $\Delta_{s_3\tau}(Z_4)$, note that $z_{44}=0$.)
\end{example}

\section{A recurrence relation}\label{s.recursion}

Now we come to the main theorem concerning characters of flagged diagram modules.  We continue to assume $E_\bull$ is standard, and from now on, we also assume $\kk$ is a field.  Given a diagram $D$ and a flag $E_\bull$, recall that $\fS_D = \tchar_T(E_\bull^D)$ is the corresponding character, and that our conventions ensure $\fS_D$ is a polynomial in $x_1,\ldots,x_n$.

The theorem requires a further condition on diagrams.

\begin{definition}
A diagram $D$ is \define{clear} if for every $k$, either $D$ is $k$-full or $D$ has a descent at $k$.
\end{definition}

\begin{example}\label{ex.rothe}
Given a permutation $w$, the Rothe diagram $D(w)$ consists of the boxes $(i,j)$ such that $j<w(i)$ and $i<w^{-1}(j)$.  (Writing the permuation matrix by placing dots in positions $(i,w(i))$, the boxes of $D(w)$ are those remaining after eliminating the boxes below and right of the dots.)  Rothe diagrams of permutations are clear: for each $k$, either $w(k)>w(k+1)$ (so $k$ is a \define{descent} of $w$) or $w(k)<w(k+1)$.  In the latter case, $D(w)$ is readily seen to be $k$-full.

On the other hand, if $k$ is a descent of $w$, then it is a descent of $D(w)$, witnessed by the border cell $(k,j)$ where $j=w(k+1)$ (cf.~\cite[Definition~4.4]{fgrs}).
\end{example}

\begin{example}\label{ex.repeat}
If $D$ is clear, and $D'$ is obtained by repeating or omitting some columns of $D$, then $D'$ is also clear.  
In particular, for a permutation $w$, the diagram $m\cdot D(w)$ (obtained by repeating each column of the Rothe diagram $m$ times) is clear.  (Such diagrams are not always diagrams of permutations; for example, $2\cdot D(2143)$ is not equal to $D(w)$ for any $w$.)
\end{example}

\begin{example}
On the other hand, this imposes a nontrivial restriction---not every diagram is clear.  The smallest non-example is the single-column diagram $D = \big[ \{1,3\} \big]$; it is neither $1$-full nor has a descent at $1$.  A more interesting one is shown below; it is neither $3$-full nor has a descent at $3$.
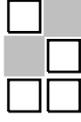
\begin{center}
\pspicture(0,0)(30,30)

\emptygraybox(0,10)

\blankbox(0,20) 
\blankbox(0,0)

\emptygraybox(10,20)

\blankbox(10,10)
\blankbox(10,0)  

\endpspicture
\end{center}
(In the terminology of \cite{rs}, this diagram is not northwest, but it is $\%$-avoiding.)
\end{example}

\begin{theorem}\label{t.recur}
Suppose $D$ is a clear diagram.  Then
\begin{equation}\label{e.recur}
  \fS_D = \Rop_1\left( \fS_D \right) + \sum_{k\in \mathrm{Des(D)}} x_k\cdot \Rop_{k+1}\left( \fS_{s_kD} \right),
\end{equation}
where $\mathrm{Des}(D)$ denotes the set of descents of the diagram $D$.
\end{theorem}

\begin{proof}
The $\Rop$-operators produce a filtration of $E_\bull^D$ by quotients: assuming $D$ has no boxes below row $n$, we have
\[
  E_\bull^D = \Rop_{n+1}E_\bull^D \tto \Rop_nE_\bull^D \tto \cdots \tto \Rop_2E_\bull^D \tto \Rop_1E_\bull^D.
\]
It follows that
\begin{equation}\label{e.recur1}
  \fS_D = \tchar_T\left( \Rop_1E_\bull^D \right) + \sum_{k=1}^n \tchar_T\left( \ker(\Rop_{k+1}E_\bull^D \tto \Rop_kE_\bull^D) \right).
\end{equation}

Since $D$ is clear, it is $k$-full whenever $k$ is not a descent; by Lemma~\ref{l.k-full}, $\ker(\Rop_{k+1}E_\bull^D \tto \Rop_kE_\bull^D)=0$ in this case.  So the RHS of \eqref{e.recur1} ranges over $k\in\mathrm{Des}(D)$.

When $k$ is a descent, Lemma~\ref{l.k-descent} says
\[
  \ker(\Rop_{k+1}E_\bull^D \tto \Rop_kE_\bull^D) \isom \ker(\Rop_{k+1}E_k \tto \Rop_kE_k) \otimes \Rop_{k+1}E_\bull^{s_kD},
\]
and by Lemma~\ref{l.rk-character}, this has character $x_k\cdot \Rop_{k+1}\fS_{s_kD}$, as claimed.
\end{proof}

The formula of Theorem~\ref{t.recur} can be rephrased as
\begin{equation}\label{e.recur2}
 \fS_D = \frac{1}{1-\Rop_1} \sum_{k\in \mathrm{Des(D)}} x_k\cdot \Rop_{k+1}\left( \fS_{s_kD} \right),
\end{equation}
cf.~\cite[(3.1)]{nst2}.  This version expresses $\fS_D$ in terms of lower-degree characters $\fS_{s_kD}$ on the right-hand side.  (Since $\Rop_1$ is a locally nilpotent operator, the action of $(1-\Rop_1)^{-1} = 1+\Rop_1 + \Rop_1^2 + \cdots$ involves only finitely many terms.)

\begin{remark}\label{r.nst-compare}
Comparing \eqref{e.recur} with that of \cite[(3.1)]{nst2}, the indices of the operators $\Rop_k$ are shifted by one.  As explained in \cite[Remark~3.6]{nst2}, the formulas agree when $D=D(w)$ is a Rothe diagram.  In fact, Lemma~\ref{l.k-full} implies that whenever $s_kD$ is $k$-full, we have $\Rop_{k+1}\fS_{s_kD} = \Rop_k\fS_{s_kD}$, so the index shift does not change the formula in this case.  For Rothe diagrams, $s_kD(w)=D(ws_k)$ is $k$-full whenever $k\in \mathrm{Des}(w)$---however, this fails for general $D$.
\end{remark}

The formula \eqref{e.recur2} becomes a recurrence relation determining $\fS_D$ under a further hypothesis.
\begin{definition}\label{d.transparent}
The set of \define{transparent} diagrams is defined inductively as follows:
\begin{itemize}
\item (base) The empty diagram $D=\emptyset$ is transparent.

\item (step) A diagram $D$ is transparent if $D$ is clear and $s_kD$ is transparent for every $k\in \mathrm{Des}(D)$.
\end{itemize}
\end{definition}

\begin{corollary}
If $D$ is transparent, then the character $\fS_D$ is determined by the recurrence relation \eqref{e.recur} together with the fact that $\fS_\emptyset=1$.
\end{corollary}

\begin{proof}
On the RHS of \eqref{e.recur}, the term $\Rop_1(\fS_D)$ involves fewer variables than $\fS_D$, and the terms $\Rop_{k+1}(\fS_{s_kD})$ are of lower degree than $\fS_D$.  So both can be assumed to be known, by induction on degree and number of variables.  (Alternatively, one can use the reformulation \eqref{e.recur2} and induction only on degree.)
\end{proof}

The recurrence relation implies a generalization of \cite[Theorem~1.1]{nst2}, giving a formula for the character of a transparent diagram $D$ in terms of reduced words.

\begin{definition}
Let $\ell$ be the number of boxes in $D$.  We say $(i_1,\ldots,i_\ell)$ is a \define{reduced word} for $D$ if $s_{i_1}\cdots s_{i_\ell}D = \emptyset$.  That is, there is a sequence of diagrams $D=D_\ell, \ldots, D_1, D_0=\emptyset$ such that $i_j \in \mathrm{Des}(D_j)$ and $D_{j-1} = s_{i_j}D_j$.

We write $\mathrm{Red}(D)$ for the set of all reduced words of $D$.
\end{definition}

When $D=D(w)$ is the Rothe diagram of a permutation, a reduced word for $D$ is the same as a reduced word for $w$.

\begin{corollary}\label{c.redword}
Let $D$ be a transparent diagram.  Then
\begin{equation}\label{e.redwordcor}
 \fS_D = \sum_{(i_1,\ldots,i_\ell)\in\mathrm{Red}(D)} \Zop x_{i_\ell} \Rop_{i_\ell+1} \cdots \Zop x_{i_1} \Rop_{i_1+1}(1),
\end{equation}
where $\Zop=1/(1-\Rop_1)$.
\end{corollary}

\begin{remark}
Compared with the formula of \cite[Theorem~1.1]{nst2}, the above formula has the indices of the Bergeron-Sottile operators $\Rop_k$ shifted by one, for the reasons explained in Remark~\ref{r.nst-compare}.  This leaves the formula unchanged when $D=D(w)$, but not for general diagrams.
\end{remark}

In fact, the recurrence relation determines $\fS_D$ for the slightly wider class of {\it translucent} diagrams.
\begin{definition}\label{d.translucent}
A diagram is \define{translucent}\footnote{A {\it clear} diagram is one which does not obstruct the recursive step \eqref{e.recur}; a {\it transparent} diagram is one for which the recursion passes all the way through to the base case $D=\emptyset$; a {\it translucent} diagram is not quite transparent, but lets through enough recursion to reach single-column base cases.} if it belongs to the following inductively defined set:
\begin{itemize}
\item (base) A diagram with a single (possibly empty) column $D=[\ba]$ is translucent.

\item (step) A diagram $D$ with at least two columns is translucent if $D$ is clear and $s_kD$ is translucent for every $k\in \mathrm{Des}(D)$.
\end{itemize}
\end{definition}

\begin{corollary}
If $D$ is translucent, then the character $\fS_D$ is determined by the recurrence relation \eqref{e.recur} together with the fact that
\begin{equation}
  \fS_\ba= \sum_{\bb\leq\ba} x_{b_1} x_{b_2} \cdots x_{b_r},
\end{equation}
where $\bb=\{b_1<b_2<\cdots<b_r\}$.
\end{corollary}

(The recurrence can be repackaged into a sum over appropriately defined ``reduced words'' for translucent diagrams, as in Corollary~\ref{c.redword}, but the statement is less elegant and we leave it to the interested reader.)

\begin{example}\label{ex.tr-diag}
The diagram $D$ shown below, is translucent, but not transparent.  
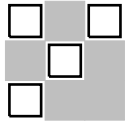
\begin{center}
\pspicture(0,0)(30,30)

\emptygraybox(0,10)

\blankbox(0,20) 
\blankbox(0,0)

\emptygraybox(10,20)
\emptygraybox(10,0)

\blankbox(10,10)


\emptygraybox(20,10)
\emptygraybox(20,0)

\blankbox(20,20)  

\endpspicture
\end{center}
Using the recursion, one computes $\fS_D = x_1^3 x_2 + x_1^3 x_3 + x_1^2 x_2^2 + x_1^2 x_2 x_3$.  (Alternatively, one can observe $D=D' \cup D''$, where $D' = \big[ \{1,3\},\, \{1\}\big]$ and $D''=\big[\{2\}\big]$ occupy disjoint rows, so $\fS_D = \fS_{D'}\cdot\fS_{D''} = (x_1^2 x_2 + x_1^2 x_3)(x_1+x_2)$.)
\end{example}

\begin{remark}
The translucent diagram of the previous example is not \%-avoiding, nor is any permutation of its columns.  On the other hand, the diagram of Remark~\ref{r.cubic} is \%-avoiding, but not translucent.
\end{remark}

\begin{remark}
As observed in Example~\ref{ex.repeat}, when $D$ is clear, so is $m\cdot D$.  Repeating columns does not preserve the properties of being transparent or translucent, however.  For example $D=D(2\,1\,4\,5\,3)$ is transparent, but $2\cdot D$ is not even translucent.
\end{remark}

\begin{remark}
Similarly, the $\mathsf{Boxcomp}$ operation (studied extensively in \cite{rs}) fails to preserve the class of translucent diagrams.  For instance, considering the diagram $D$ of Example~\ref{ex.tr-diag} inside the $3\times 3$ box, its complement
\begin{center}
\pspicture(0,0)(30,30)

\emptygraybox(0,20)
\emptygraybox(0,0)

\blankbox(0,10)


\emptygraybox(10,10)

\blankbox(10,20)  
\blankbox(10,0)

\emptygraybox(20,20)

\blankbox(20,10)
\blankbox(20,0)  

\endpspicture
\end{center}
is not translucent.
\end{remark}

\begin{remark}
As operators on transparent diagrams, the $s_k$ satisfy $s_k s_{k+1} s_k = s_{k+1} s_k s_{k+1}$ and $s_i s_j = s_j s_i$ when $|i-j|>1$.  So they determine an action of the braid monoid $B_+$.  In particular, one can perform commutations and braid moves on a reduced word for $D$ to produce another reduced word.  However, in contrast to the situation with reduced words for permutations, the commutation and braid moves do not connect every pair of reduced words for a diagram $D$.  For instance, if $D = \big[ \{1\}, \{2\} \big]$, its reduced words are $(1,\,1)$ and $(1,\,2)$.
\end{remark}

\section{Examples and applications}\label{s.apps}

\subsection{Schubert polynomials and Schubert modules}\label{ss.schub}

The \define{Schubert polynomial} $\fS_w$ associated to a permutation is usually defined via descending induction on the length of $w$; see the next subsection.  However, thanks to \cite[Theorem~1.1 and (3.1)]{nst2}, we have an alternative characterization via an increasing recursive formula: $\fS_{\mathrm{id}}=1$, and
\begin{equation}\label{e.nst-schub}
  \fS_w = \Rop_1(\fS_{w}) + \sum_{k\in\mathrm{Des}(w)} x_k \cdot \Rop_{k+1}( \fS_{ws_k} ).
\end{equation}
For now, we may take this as a definition of $\fS_w$.

On the other hand, the Rothe diagram $D(w)$ of a permutation $w$ is transparent, by induction on the length of $w$: As explained in Example~\ref{ex.rothe}, $D(w)$ is clear.  And one checks that $s_kD(w) = D(ws_k)$---this is the motivating case for the definition of $s_kD$---so it is transparent by the inductive assumption.  It follows that the characters $\fS_{D(w)}$ are characterized uniquely by an identical recursive formula (cf.~\eqref{e.recur} above):
\begin{equation}
  \fS_{D(w)} = \Rop_1(\fS_{D(w)}) + \sum_{k\in\mathrm{Des}(D(w))} x_k \cdot \Rop_{k+1}( \fS_{D(ws_k)} ).
\end{equation}
(As noted before, the descent sets of $D(w)$ and of $w$ coincide---$\mathrm{Des}(D(w))=\mathrm{Des}(w)$---so this really is the same formula.)  Thus we recover the theorem of Kra\'skiewicz and Pragacz:
\begin{corollary}
For a Rothe diagram $D(w)$, the character of the flagged module $E_\bull^{D(w)}$ is equal to the Schubert polynomial $\fS_w$.
\end{corollary}

\begin{example}
The diagram $D = \big[\{2,3\},\,\{2,3,5\},\,\{3\} \big]$ of our running example is equal to $D(1\,4\,6\,2\,5\,3)$, with empty columns removed.  So we have
\begin{align*}
\fS_D = \fS_{1\,4\,6\,2\,5\,3} &= x_1^3 x_2^2 x_3 + x_1^3 x_2^2 x_4 + x_1^3 x_2^2 x_5 + x_1^3 x_2 x_3^2 + x_1^3 x_2 x_3 x_4 + x_1^3 x_2 x_3 x_5 \\
& \quad + x_1^3 x_3^2 x_4 + x_1^3 x_3^2 x_5 + x_1^2 x_2^3 x_3 + x_1^2 x_2^3 x_4 + x_1^2 x_2^3 x_5 + 2 x_1^2 x_2^2 x_3^2 \\
& \quad + 2 x_1^2 x_2^2 x_3 x_4 + 2 x_1^2 x_2^2 x_3 x_5 + x_1^2 x_2 x_3^3 + 2 x_1^2 x_2 x_3^2 x_4 + 2 x_1^2 x_2 x_3^2 x_5 \\
& \quad + x_1^2 x_3^3 x_4 + x_1^2 x_3^3 x_5 + x_1 x_2^3 x_3^2 + x_1 x_2^3 x_3 x_4 + x_1 x_2^3 x_3 x_5 + x_1 x_2^2 x_3^3 \\
& \quad + 2 x_1 x_2^2 x_3^2 x_4 + 2 x_1 x_2^2 x_3^2 x_5 + x_1 x_2 x_3^3 x_4 + x_1 x_2 x_3^3 x_5 + x_2^3 x_3^2 x_4 \\ 
& \quad + x_2^3 x_3^2 x_5 + x_2^2 x_3^3 x_4 + x_2^2 x_3^3 x_5.
\end{align*}
Adding up the coefficients, we find that $\dim E_\bull^D = 38$.
\end{example}

\subsection{Divided difference operators}

Schubert polynomials are typically defined via a different recursion, but one which turns out to be closely related.  The \define{divided difference operator} $\partial_k$ is defined on polynomials by
\[
  \partial_k f = \frac{f - s_kf}{x_k-x_{k+1}}=\frac{f(\ldots,x_k,x_{k+1},\ldots) - f(\ldots,x_{k+1},x_k,\ldots)}{x_k-x_{k+1}}.
\]
The Schubert polynomials $\fS_w$ for $w\in S_n$ are then defined inductively by
\begin{equation*}
  \fS_{[n,n-1,\ldots,1]} = x_1^{n-1} x_2^{n-2} \cdots x_{n-1} 
\end{equation*}
and
\begin{equation}\label{e.schub-dd}
   \partial_k \fS_w = \begin{cases} \fS_{ws_k} &\text{if } k\in \mathrm{Des}(w); \\ 0 &\text{otherwise.}\end{cases} 
\end{equation}
As we have just seen, $\fS_D=\fS_w$ when $D=D(w)$ is the Rothe diagram of a permutation, so in this case we have
\begin{equation}\label{e.schub-dd2}
  \partial_k \fS_D = \begin{cases} \fS_{s_kD} &\text{if } k\in \mathrm{Des}(D); \\ 0 &\text{otherwise.}\end{cases}
\end{equation}

The naive analogue of \eqref{e.schub-dd2} is false for more general diagrams---it is easy to find small counterexamples, e.g., for $D=\big[\{1\},\{2\}\big]$.  It fails even for multiples of Rothe diagrams, $D=m\cdot D(w)$: for $w=[2,1,4,5,3]$ and $m=2$, we have $\partial_1\fS_D \neq \fS_{s_1D}$.

On the other hand, consider the \define{trimming operator} $\Top_k$, introduced in \cite{nst} and defined by
\[
  \Top_k f = \frac{\Rop_{k+1}f - \Rop_k f}{x_k}.
\]
Lemmas~\ref{l.rk-character}, \ref{l.k-full}, and \ref{l.k-descent} imply the following:
\begin{corollary}\label{c.top}
For any diagram $D$, we have
\[
  \Top_k\fS_D = \begin{cases} \Rop_{k+1}\fS_{s_kD} & \text{if } k\in\mathrm{Des}(D); \\ 0 &\text{if } D\text{ is $k$-full.} \end{cases}
\]
\end{corollary}

When $D$ is clear, the corollary gives a formula for $\Top_k\fS_D$ for all $k$.  Since $\Top_k = \Rop_k\partial_k = \Rop_{k+1}\partial_k$ this says that \eqref{e.schub-dd2} holds after applying $\Rop_{k+1}$ to both sides.

\subsection{Graded algebras}

Given a diagram $D$ and a free $\kk$-module $E$, consider the graded module
\begin{equation}
 \cR(D,E) = \bigoplus_{m\geq 0} E^{mD}.
\end{equation}
(Here, as before, $mD$ is the diagram in which each column of $D$ repeats $m$ times.)  This is a graded subalgebra of $\kk[Z]$.

Similarly, given a flagged, free $\kk$-module $E_\bull$, we have
\begin{equation}
 \cR(D,E_\bull) = \bigoplus_{m\geq 0} E_\bull^{mD},
\end{equation}
a graded subalgebra of the polynomial ring $\kk[Z_\bull]$.  The homomorphisms $\Rop_{k+1}E_\bull^{mD} \tto \Rop_{k}E_\bull^{mD}$ determine graded ring homomorphisms
\begin{equation}\label{e.cr-filtration}
  \cR(D,E_\bull) \tto \cR(D,\Rop_nE_\bull) \tto \cdots \tto \cR(D,\Rop_2E_\bull) \tto \cR(D,\Rop_1E_\bull) .
\end{equation}

\begin{corollary}
If $D$ is clear, the filtration \eqref{e.cr-filtration} has associated graded
\[
  \mathrm{gr} \cR(D,E_\bull) = \cR(D,\Rop_1E_\bull)\oplus\bigoplus_{k\in\mathrm{Des}(D)} K_k \otimes s_k\cR(D,\Rop_{k+1}E_\bull),
\]
where $K_k = \ker(E_{k+1} \to E_k)$ as above, and $s_k\cR(D,E_\bull) = \bigoplus_{m\geq0} E_\bull^{s_k mD}$.
\end{corollary}

The alegbra $\mathrm{gr} \cR(D,E_\bull)$ retains its grading by $m$, and with respect to this grading, the ring $\cR(D,\Rop_1E_\bull)$ is a graded subalgebra (as well as a quotient).  That is, passing to the associated graded produces a canonical section of the projection $\cR(D,E_\bull) \tto \cR(D,\Rop_1E_\bull)$.

The projective variety $\Proj \cR(D,\Rop_kE_\bull)$ can be identified with Magyar's {\it configuration variety}, and the above discussion determines a canonical degeneration.  Describing these degenerations in greater detail should be an interesting project.

\subsection{Positivity}

Schubert polynomials $\fS_w(x)$ form a $\Z$-linear basis of the polynomial ring, and one often wants to know about the Schubert expansion of a given polynomial $f(x)$---especially when the coefficients in this expansion are nonnegative.  In our context, where $f=\fS_D$ is the character of a flagged Schur module $E_\bull^D$, this question was investigated by Watanabe \cite{watanabe-positive}.

Given a diagram $D$, we have
\begin{equation}\label{e.cwD}
  \fS_D = \sum_w c^w_D\, \fS_w,
\end{equation}
where the sum is over all permutations, and the coefficients $c^w_D$ are integers.  Even for translucent and $\%$-avoiding diagrams, $c^w_D$ may be negative.  For example, $D = \big[ \{1,3\}\big]$ is both translucent and $\%$-avoiding, and $\fS_D = -\fS_{3\,1\,2} + \fS_{2\, 1\, 4\, 3}$.

On the other hand, extensive computer experimentation suggests the following:
\begin{conjecture}
When $D$ is transparent, the coefficients $c^w_D$ defined by \eqref{e.cwD} are nonnegative.
\end{conjecture}

For example, consider the following diagram $D$:
\begin{center}
\pspicture(0,0)(50,40)

\blankbox(0,30)  
\blankbox(0,20)  
\blankbox(0,10)  
\blankbox(0,0)   

\emptygraybox(10,30)  
\emptygraybox(10,20)  
\emptygraybox(10,10)  
\blankbox(10,0)      

\emptygraybox(20,30)  
\blankbox(20,20)      
\emptygraybox(20,10)  
\emptygraybox(20,0)   

\emptygraybox(30,30)  
\emptygraybox(30,20)  
\blankbox(30,10)      
\emptygraybox(30,0)   

\emptygraybox(40,30)  
\blankbox(40,20)      
\blankbox(40,10)      
\blankbox(40,0)       

\endpspicture
\end{center}
Using Julia (see \cite{schubmods}), one can quickly check that $D$ is transparent, and $E_\bull^D$ has character
\[
 \fS_D = \fS_{2\, 5\, 7\, 4\, 1\, 3\, 6} + \fS_{2\, 7\, 4\, 5\, 1\, 3\, 6} + \fS_{2\, 6\, 4\, 7\, 1\, 3\, 5} + \fS_{2\, 5\, 6\, 7\, 1\, 3\, 4}.
\]

\appendix

\section{Partial flags}\label{app.parabolic}

One can consider more general partial flags $E_\bull$, where $\rk E_i = d_i$ for some increasing sequence $\bd$ of integers, $0=d_0<d_1<\cdots<d_n=N$.  Here we will briefly review the changes required to accommodate this setup.  As we will see, most of the constructions carry over easily, but the stronger results---for example, the simple characterization of $\fS_D$ via the recursion \eqref{e.recur}---do not.

\subsection{Flagged modules}

The constructions of $E_\bull^\ba$, $TE_\bull^D$, and $E_\bull^D$ go much as before.

Assume $E_\bull$ is free, with $\rk E_i = d_i$, and fix bases $z_{i,j}$ for $E_i$ so that the projection $E_i \to E_{i-1}$ is given by
\[
  z_{i,j} \mapsto \begin{cases} z_{i-1,j} & \text{if }j\leq d_{i-1}; \\ 0 & \text{otherwise.} \end{cases}
\]
This gives us the notion of a \define{standard flag}.

A basis for $E_\bull^\ba$ is provided by $z_{\ba,\bb}$ for $\bb\leq\ba$ as before, but now we redefine this by
\[
  \bb \leq \ba \quad \text{ if } \quad b_i \leq d_{a_i} \text{ for all }i.
\]
A column-strict flagged filling $\tau$ of a diagram $D$ is one where the entries in row $i$ are bounded above by $d_i$.  As before, $TE_\bull^D$ has a basis $z_\tau$ ranging over such fillings.

The $D$-exchange relations can be determined by mapping to a polynomial ring in variables
\begin{align*}
&  z_{11}, z_{12}, \ldots, z_{1,d_1}, \\
&  z_{21}, z_{22}, \ldots, z_{2,d_1}, \ldots, z_{2,d_2},\\
&  \vdots \\
&  z_{n1}, z_{n2}, \ldots, z_{n,d_1}, \ldots, z_{n,d_2}, \ldots, z_{n,d_n}.
\end{align*}
Let $Z_\bull = (z_{ij})_{d_i\leq j}$ be the $n\times d_n$ block lower-triangular matrix of the form suggested above.  The homomorphism $TE_\bull^D \to \kk[Z_\bull]$ is defined as before, and one has a pushout diagram analogous to that of \eqref{e.pushout-flag}.

Group actions are also as before, except now the group acting is the parabolic subgroup $P\subset GL(E)$ preserving $E_\bull$.  In our chosen basis, $P$ is block lower-triangular, with block sizes $c_i = d_i-d_{i-1}$.

\subsection{Bergeron-Sottile operators}

From now on we assume $\kk$ is a field, and $E_\bull$ is standard.  We use the same definition for operators on flags:
\[
  \Rop^\bd_kE_i = \begin{cases} E_i & \text{if } i< k; \\ E_{i-1} & \text{if } i\geq k. \end{cases}
\]
The corresponding action on polynomial rings is different:
\[
 \Rop^\bd_k(x_i)=\begin{cases} x_i &\text{if }i<d_k, \\ 0 &\text{if }d_k\leq i<d_{k+1}, \\ x_{i-c_{k+1}} &\text{if } i\geq d_{k+1}. \end{cases}
\]
(Recall that $c_k = d_k-d_{k-1}$.)  When the differences $c_k$ are all equal, say $c_k=m$ for all $m$, we have $\Rop_k^\bd = \Rop_k^m$, cf.~\cite{nst,nst2}.

The analogous projections of polynomial rings $\kk[Z_{k+1}] \tto \kk[Z_k]$ set $z_{k,j}=0$ for $d_{k-1}< j\leq d_k$.  For instance, with $\bd = \{1,2,5,6,7,\ldots\}$ and $k=3$, we have
\begin{align}
  \kk[Z_{4}] & \to \kk[Z_3] \label{e.z4to3} \\
{\left(\begin{array}{cccccccc}
z_{11}  \\ 
z_{21} & z_{22}  \\ 
z_{31} & z_{32} & \bm{z_{33}} & \bm{z_{34}} & \bm{z_{35}}  \\ 
z_{41} & z_{42} & z_{43} & z_{44} & z_{45} & 0 \\ 
z_{51} & z_{52} & z_{53} & z_{54} & z_{55} & z_{56} & 0  \\ 
\vdots &  &  &  & & &  & \ddots 
\end{array}\right)} & \to
{\left(\begin{array}{cccccccc}
z_{11}  \\ 
z_{21} & z_{22}  \\ 
z_{31} & z_{32} & \bm{0} & \bm{0} & \bm{0}  \\ 
z_{41} & z_{42} & z_{43} & z_{44} & z_{45} & 0 \\ 
z_{51} & z_{52} & z_{53} & z_{54} & z_{55} & z_{56} & 0  \\ 
\vdots &  &  &  & & &  & \ddots 
\end{array}\right)}. \nonumber
\end{align}

We define $\bb\leq_k \ba$ analogously to before, to mean $b_i\leq d_{a_i}$ for all $i$, and $b_i\leq d_{a_i-1}$ whenever $a_i\geq k$.  So $\Rop_k^\bd E_\bull^\ba$ has a basis of $z_{\ba,\bb}$ for $\bb\leq_k \ba$.

\subsection{Kernels of Bergeron-Sottile operators}

In our setting of partial flags indexed by $\bd$, there are no simple analogues of the recursive formulas from \S\ref{s.recursion}.  This is because the kernels computed in Lemma~\ref{l.k-full} and \ref{l.k-descent} become more complicated when $d_i\neq i$.

For single-column modules, the obvious analogue of Lemma~\ref{l.Na} holds, for the same reason:
\begin{lemma}
The kernel $N^{\ba}_k = \ker( \Rop^\bd_{k+1}E_\bull^\ba \tto \Rop^\bd_k E_\bull^\ba )$ is generated by $z_{\ba,\bb}$ such that $\bb \leq_{k+1} \ba$, and for some $i$ we have $a_i=k$ and $d_{k-1}<b_i\leq d_k$.  (That is, $\bb\not\leq_{k}\ba$.)
\end{lemma}

Here we see the first contrast to the case $d_i=i$: the projection $\Rop^\bd_{k+1}\tto \Rop^\bd_k$ need not be an isomorphism for $k$-full diagrams.  Let $K_k = \ker(\Rop^\bd_{k+1}E_k \to \Rop^\bd_kE_k) = \ker(E_k\tto E_{k-1})$.

\begin{lemma}
Let $N^{\ba}_k = \ker( \Rop^\bd_{k+1}E_\bull^\ba \tto \Rop^\bd_k E_\bull^\ba )$ as above.
\begin{enumerate}[itemsep=5pt]
\item If $k\not\in\ba$, then $N_k^{\ba}=0$, that is, $\Rop^\bd_{k+1}E_\bull^\ba \tto \Rop^\bd_kE_\bull^\ba$ is an isomorphism.

\item If $\ba \supseteq \{k,k+1\}$, then
\[
  N_k^{\ba} \isom \exterior^2 K_k\otimes \Rop^\bd_{k+1}E_\bull^{\ba'},
\]
where $\ba'=\ba\setminus\{k,k+1\}$.

\item If $\ba\cap\{k,k+1\}=\{k\}$, then
\[
  N_k^{\ba} \isom K_k\otimes \Rop^\bd_{k+1}E_\bull^{\ba'},
\]
where $\ba'=\ba\setminus\{k\}$.
\end{enumerate}
\end{lemma}

These statements follow directly from the previous lemma.  Note that cases (i) and (ii) comprise the possibilities for a $k$-full column; when case (ii) holds and $\dim K_k \geq 2$, $N_k^\ba$ can be nonzero.

\begin{example}
Take $\bd = \{ 1,2,5,6,7 \}$ as in \eqref{e.z4to3} above, and $\ba=\{ 3,4 \}$, so for $k=3$ we are in case (ii) of the lemma.  The  elements
\[
  z_{\{3,4\},\{3,4\}} = \left|\begin{array}{cc} z_{33} & z_{34}  \\ z_{43} & z_{43}  \end{array}\right|,\quad 
   z_{\{3,4\},\{3,5\}} = \left|\begin{array}{cc} z_{33} & z_{35}  \\ z_{43} & z_{45}  \end{array}\right|,\quad 
    z_{\{3,4\},\{4,5\}} = \left|\begin{array}{cc} z_{34} & z_{35}  \\ z_{44} & z_{45}  \end{array}\right| 
\]
are nonzero in $\Rop_{4}E_\bull^\ba$, and they all map to $0$ under the map to $\Rop_3 E_\bull^\ba$.
\end{example}

Similarly, when $k$ is a descent of a diagram $D$, the argument of Lemma~\ref{l.k-descent} shows that $N_k^D$ is isomorphic to the image of the homomorphism
\begin{equation}
  K_k \otimes \Rop_{k+1} E_\bull^{s_kD}  \to \kk[Z_\bull]
\end{equation}
given by multiplication, regading both $K_k$ as the span of $z_{k,j}$ for $d_{k-1}<j\leq d_k$, and $E_\bull^{s_kD}$ as the span of appropriate products of determinants $\Delta_\tau$, so both are given as subspaces of $\kk[Z_\bull]$.  When $\dim K_k \geq 2$, this multiplication map can have a kernel.  This happens already in very simple cases.

\begin{example}
Consider the diagram $D = \big[\{1\},\{1\}\big]$, and suppose $d_1=2$ and $d_2=3$, i.e., $\dim E_1 = 2$ and $\dim E_2 = 3$.  Then $E_\bull^D \isom \Sym^2 E_1$.  The diagram $D$ has a descent at $k=1$, and $\Rop_1 E_1 = 0$, so 
\[
K_1 = E_1 \quad \text{ and }\quad N_1^D = E_\bull^D = \Sym^2 E_1.
\]
After erasing an empty column, $s_1D = \big[ \{2\} \big]$, so $E_\bull^{s_1D} = E_2$, and $\Rop_2 E_2 = E_1$.  So the map $K_1 \otimes  \Rop_2 E_\bull^{s_1 D} \tto N_1^D$ is the projection
\[
  E_1 \otimes E_1 \to \Sym^2 E_1,
\]
and not an isomorphism.
\end{example}



%

\end{document}